\documentclass[11pt]{article}%
\usepackage{mathrsfs}
\usepackage[letterpaper]{geometry}
\usepackage{natbib}
\usepackage{color}
\usepackage{hyperref}
\usepackage{enumerate}
\usepackage{amsthm,amsmath,amsfonts}
\usepackage{graphicx,tikz}
\usetikzlibrary{arrows,automata}

\newcommand{\R}{\mathbb{R}}
\newcommand{\Z}{\mathbb{Z}}
\newcommand{\E}{\mathbb{E}}
\newcommand{\D}{\mathbb{D}}
\newcommand{\X}{\mathbb{X}}
\newcommand{\tS}{\tilde {\cal S}}
\newcommand{\tT}{\tilde {\cal T}}

\DeclareMathOperator{\diag}{diag}

\newtheorem{lemma}{Lemma}
\newtheorem{assumption}{Assumption}
\newtheorem{theorem}{Theorem}
\newtheorem{proposition}{Proposition}
\newtheorem{corollary}{Corollary}

\numberwithin{equation}{section}

\newcommand{\gaonote}[1]{\textcolor{black}{#1}}
\newcommand{\tonnote}[1]{\textcolor{black}{#1}}
\newcommand{\jdnote}[1]{\textcolor{black}{#1}}

\begin{document}

\begin{center}
{\Large \textbf{Validity of heavy-traffic steady-state
approximations in many-server queues with abandonment}}
\end{center}

\author{}
\begin{center}
  J. G. Dai\,\footnote{School of ORIE, Cornell University, Ithaca, NY
    14850; jim.dai@cornell.edu; on leave from Georgia Institute of
    Technology}, A. B. Dieker\,\footnote{H. Milton Stewart School of
    Industrial and Systems Engineering, Georgia Institute of
    Technology, Atlanta, GA 30332; ton.dieker@isye.gatech.edu}, and Xuefeng
  Gao\,\footnote{Department of Systems Engineering and Engineering Management, The Chinese University of Hong Kong,  Shatin, N. T., Hong Kong; gxf1240@gmail.com}
\end{center}
\begin{center}
\today
\end{center}

\begin{abstract}
We consider $GI/Ph/n+M$ parallel-server systems with a renewal
arrival process,  a phase-type service time distribution, $n$
homogenous servers, and an exponential patience time distribution
with positive rate. We show that in the Halfin-Whitt regime, the
  sequence of stationary distributions
corresponding to the normalized state processes is tight.
 As a consequence, we
establish an interchange of heavy traffic and steady state limits
for $GI/Ph/n+M$ queues.
\end{abstract}


\section{Introduction}

Parallel server queueing systems with customer abandonment serve as a
building block for modeling complex service systems such as customer call
centers. \cite{ganskoolemandelbaum} argued that best practice
companies should operate a service system in a quality- and
efficiency-driven (QED) regime, also known as the Halfin-Whitt regime.
In this regime, the service system
provides high-quality service and at the same time achieves high
server utilization. This is possible only when there is a large
number of servers working in parallel on a common pool of customers.

This paper focuses on a subset of $GI/GI/n+M$ queueing systems. In
such a queueing system, there are $n$ identical servers serving
customers whose interarrival times are i.i.d.\ following a general
distribution (the first $GI$) and whose service times are i.i.d.\
following a general distribution (the second $GI$). In addition,
customers' patience times are i.i.d.\ following an exponential
distribution ($+M$): when a customer's waiting time in queue exceeds
his patience time, the customer abandons the system without service.

Despite decades of research, there is still no general analytical or
numerical tool to efficiently and accurately predict the steady-state
performance of a $GI/GI/n+M$ system. Computer simulation is often the
only remaining tool available, but it can be slow when the number of
servers is large, particularly when the system is in the QED regime.

Recently, \cite{DaiHe13} proposed diffusion models to approximate a
$GI/Ph/n+M$ system when the service time distribution is of
phase-type.  Their approximations are rooted in many-server heavy
traffic limit theorems proved in \cite{DaiHeTezcan10}.  The
numerical examples in \cite{DaiHe13} demonstrate that the
steady-state performance of the diffusion model provides a
remarkably accurate estimate for the steady-state performance of the
corresponding queueing system, even when the number of servers is
moderate. In this paper, we provide a justification for
the diffusion approximation procedure in \cite{DaiHe13}.

We now describe our results and contributions more precisely.  When
the number of servers $n$ is fixed, a $GI/Ph/n+M$ system can be
represented by a certain continuous time Markov process
$\X^n=\{\X^n(t): t\ge 0\}$. Often $\X^n(t)$ converges in distribution
to $\X^n(\infty)$ as time $t\to\infty$, where the random variable
$\X^n(\infty)$ has the stationary distribution of $\X^n$.  On the
other hand, \cite{DaiHeTezcan10} shows that $\tilde \X^n$, as a
sequence of stochastic processes that are centered and scaled
versions of $\X^n$, converges in distribution to some diffusion process
$\tilde\X$ as $n\to\infty$ under a heavy traffic condition. The limit
process $\tilde \X$ is a piecewise Ornstein-Uhlenbeck (OU) process.
A similar result was proved earlier in \cite{PuhalskiiReiman00} for
$GI/Ph/n$ systems without customer abandonment.  The convergence
proved in \cite{DaiHeTezcan10}
implies that each finite $t\ge 0$, $\tilde \X^n(t)$ converges in
distribution to $\tilde \X(t)$, but, as in almost all diffusion
limits,   does not cover the case when
$t=\infty$.
 This paper proves that $\tilde
\X^n(\infty)$ converges in distribution to $\tilde \X(\infty)$ as
$n\to\infty$; see  Theorem \ref{thm:tightness}  and
 Corollary \ref{thm:interchange} in Section~\ref{sec:main-result} below.
 Often, one can also prove that $\tilde \X(t)$
converges in distribution to $\tilde \X(\infty)$ as $t\to\infty$. Formally,
our results can be stated as
\[
  \lim_{n\to\infty} \lim_{t\to\infty} \tilde \X^n(t) \stackrel{d}=
 \lim_{t\to\infty}  \lim_{n\to\infty} \tilde \X^n(t),
\]
which is known as the interchange limit theorem.

There has been a surge of interest in establishing interchange limit
theorems in the last ten years in both the conventional heavy
traffic setting and  many-server heavy traffic setting.  To prove
an interchange limit theorem, when the stationary distribution of the
limit process $\tilde \X$ is unique, it is
sufficient to prove that the
sequence of random variables $\{\tilde \X^n(\infty): n\ge 1\}$  is
tight.  \cite{GamarnikZeevi06} pioneered an approach in proving
tightness in the context of generalized Jackson networks in
conventional heavy traffic when all distributions are assumed to
have finite exponential moments. Key to their proof is the
construction of a geometric Lyapunov function.  Inspired by this
work, \cite{budhiraja2009stationary} devised an alternative method to
prove tightness along with some
other results also in the context of generalized Jackson networks,
but with the minimal two moment assumption on all distributions.
\cite{budhiraja2009stationary} did not construct a geometric
Lyapunov function, but they cleverly utilized  and sharpened a fluid
limit approach introduced in \cite{daimey95}, and their approach is
potentially more general and obtains sharper results.

Our current paper follows the approach in \cite{GamarnikZeevi06} by
constructing a geometric Lyapunov function; see Lemma
\ref{lem:abddrift} in Section~\ref{sec:Lyapunovfluid} below.  We
heavily rely on a delicate analysis of the behavior of our Lyapunov
functions applied to a fluid model for $GI/Ph/n+M$ queues. In
particular, when applied to the fluid model, we show that our Lyapunov
function decreases at a rate that is proportional to the size of the
fluid state when it is far away from origin; see part (b) of
Lemma~\ref{lem:gbetaflu}.  Because of the customer abandonment in
our model, when the waiting fluid is high, the decreasing rate in
our Lyapunov function should be expected due to abandonment.
However, when the waiting fluid is not high, but a large fluid state
is due to the huge imbalance of servers among different service
phases, the decreasing rate is by no means obvious.  Our proof
relies critically on a common quadratic function that was used in
\cite{diekergao:OU}. Using the common quadratic Lyapunov function,
the authors were able to devise a Lyapunov function and prove the
existence of the stationary distribution of the limit process
$\tilde \X $. The diffusion process $\tilde \X$ has two different
drifts in two separate regions of state space and the Lyapunov
function works well in these two regions.
It remains an open problem
whether the approach in \cite{budhiraja2009stationary} can be
adapted to our setting. \gaonote{
One major \jdnote{step in their approach is} to obtain estimates on moments of the state
 process that are uniform in both time and the scaling parameter.
In \cite{budhiraja2009stationary}, the authors relied on the uniform Lipschitz property
\tonnote{(with respect to time)} of the Skorohod map to obtain such estimates. However, in our many-server queue setting, the map introduced in Lemma~\ref{lem:Phi} does not have such property: it is Lipschitz continuous but the Lipschitz constant depends on time; see part (c) of Lemma~\ref{lem:Phi}.}

In the many-server setting without customer abandonment,
\cite{HalWhi81} is the first paper to establish a many-server
diffusion limit for the $GI/M/n$ model. In the same paper they
proved the tightness result. \cite{GamarnikMomcilovic08} proves the
tightness result where service time distribution is lattice-valued
with a finite support. \cite{gamarnik2013steady} has generalized
this result to $GI/GI/n$ queues. In a single class, multiple server
pool model, \cite{tezcan2008optimal} proved asymptotic optimality of
some routing policies and some interchange limit results.

In the many-server setting with customer abandonment,
\cite{gamarnik2012multiclass} proved a tightness result.  In their
model, customers have many classes. Each class has its own
homogeneous Poisson arrival process, exponential service time
distribution, and exponential patience time distribution with class
dependent rate. The service policy in choosing which class of
customers to serve next can be arbitrary as long as it is non-idling.
When the service policy is first-in-first-out across customer classes and the
patience rate
is independent of customer classes, their model reduces to a special
case of our model considered in this paper.
Their proof critically relies on the assumption that service time
distributions are exponential and it remains an open problem
whether their tightness result holds for non-exponential service time
distributions. Empirical study in \cite{brown2005statistical} suggests
that the service time distributions are approximately lognormal, not
exponential.

In the conventional heavy traffic setting,
\cite{gurvichinterchange2013} systematically generalized the results
in \cite{GamarnikZeevi06} to multiclass queueing networks.
\cite{katsuda2010state} proves some interchange limit results for a
multiclass single-server queue with feedback under various
disciplines. \cite{ye2010diffusion} studied interchange limit
results in a head-of-line bandwidth sharing model with two customer
classes and two servers.

The rest of the paper is organized as follows.
Section~\ref{sec:giphn+m-queu-diff} presents the background on $GI/Ph/n+M $
queues and diffusion approximations. Assuming positive Harris
recurrence, Section \ref{sec:main-result}
summarizes our main results, Theorem \ref{thm:tightness} and
Corollary~\ref{thm:interchange}.
Section \ref{sec:Lyapunovfluid} introduces our
Lyapunov function and a fluid model used
to prove Lemma~\ref{lem:abddrift}.
Section~\ref{sec:Lyapunovdiff} discusses a key
lemma, Lemma~\ref{lem:abddrift}, for proving Theorem~\ref{thm:tightness}.
In Appendix~\ref{sec:posit-harr-recurr}, we prove the positive Harris
recurrence of $GI/Ph/n+M$ queues when $n$ is fixed.
Appendix~\ref{sec:negdriftfluid} is devoted to the proof of a negative drift condition
for the fluid model.
Appendix~\ref{sec:negdriftdiff} uses this negative drift condition for the fluid
model to establish a negative drift condition for the diffusion-scaled processes.

\subsection*{Notation}
All random variables and stochastic processes are defined on a
common probability space $(\Omega,\mathcal{F},\mathbb{P})$ unless
otherwise specified.
For a positive integer  $d$, $\mathbb{R}^{d}$ denotes the
$d$-dimensional Euclidean space. Given a subset $S$ of some Euclidean space,
the space of right-continuous functions $f :\R_+
\rightarrow S$ with left limits in $(0,\infty)$ is denoted by $\D(\R_+,S)$ or simply
$\D(S)$. For a sequence of random elements $\{X_n: n =1, 2, \ldots\}$
taking values in a metric space, we write $X_n  \Rightarrow X$ to
denote the convergence of $X_n$ to $X$ in distribution. Each
stochastic process with sample paths in $\D(S)$ is considered to
be a $\D(S)$-valued random element. The space $\D(S)$ is assumed to be
endowed with the Skorohod $J_1$-topology.
Given $x\in\mathbb{R}$, we set $x^{+}=\max\{x,0\}$ and
$x^{-}=\max\{-x,0\}$.
All vectors are envisioned as column vectors. For a $K$-dimensional
vector $u$, we write $\left\vert u\right\vert$ for its Euclidean
norm. Given $ y \in \D(S)$ and $t>0$,
we set $\|y\|_{t} = \sup_{ 0 \le s \le t} |y(s)|$,
where $|\cdot|$ denotes the Euclidean norm in $S$.
Given a $K \times K$ matrix $M$, we use $M^{\prime}$ to denote
its transpose, and similarly for vector transposition.
We write $M_{ij}$ for its $(i,j)$th entry.
Let the matrix norm of $M$ be $|M|=\sum_{ij}
 |M_{i,j}|,$ where $|M_{ij}|$ is the absolute value of
$M_{ij}.$ We reserve $I$ for the $K\times K$ identity matrix and $e$
for the $K$-dimensional vector of ones.

\section{$GI/Ph/n+M$ queues and diffusion approximations}
\label{sec:giphn+m-queu-diff}
In Section \ref{sec:model-descr-prob}, we describe the $GI/Ph/n+M$
queueing model and
probabilistic assumptions and in Section \ref{sec:halfin-whitt-regime}
we state a diffusion limit under a many-server heavy traffic
condition.
\subsection{Model description and probabilistic assumptions}
\label{sec:model-descr-prob}

In a $GI/Ph/n+M$ system, there are $n$ identical servers.  We allow
some of the system parameters to change with $n$ as the number of
servers increases. For the rest of this section, we keep $n$ fixed.
Customers
arrive according to a delayed renewal process with interarrival times
given by $\{\xi^n (i): i=0, 1, 2, \ldots \}$. We assume that $\{\xi^n(i):
i=1,2, \ldots\}$ is a sequence of independent, identically
distributed (i.i.d.) random variables with a general distribution
and this sequence is  independent of
$\xi^n(0)$. Here $\xi^n(0)$ is the time that the first customer after time
$0$ is to arrive at the system.  Upon arrival, a customer enters service
immediately if an idle server is available. Otherwise, he waits in a
buffer with infinite waiting room that holds a first-in-first-out
(FIFO) queue. When a server completes a service, the server picks
the next customer from the FIFO queue if there is one
waiting. Otherwise, the server becomes idle.

The service times are i.i.d.\ random variables, following a phase-type
distribution which corresponds to the absorption time of a certain
transient continuous-time Markov chain (the $Ph$ in $GI/Ph/n+M$
notation). Specifically, let $p$ be a
$K$-dimensional nonnegative vector with entries summing up to one,
$\nu$ be a $K$-dimensional positive vector, and $P$ be a $K \times
K$ sub-stochastic matrix. We assume that the diagonal entries of $P$
are zero and $I-P$ is invertible. Consider a continuous-time Markov
chain with $K +1$ phases (or states) where phases $1, 2 \ldots, K$
are transient and phase $K +1$ is absorbing. The Markov chain has  the
initial distribution $p$ on $\{1, 2, \ldots, K\}$. Each time it enters
phase $k$, the
amount of time it stays
in phase $k$ is exponentially distributed with mean
${1}/{\nu_k}$. Each time it leaves phase $k$, the Markov chain enters
phase $j$ with probability $P_{kj}$ or enters phase $K+1$ with
probability $1-\sum_{j \le K}{P_{kj}}$. Once the Markov chain enters
state $K+1$, it stays there forever.
A phase-type random variable with parameters $(p, \nu, P)$ is defined
as the first time until the above continuous-time Markov chain reaches
state $K +1$.

As discussed in the introduction, each customer waiting in the queue
has a patience time.
The patience times are i.i.d.\ exponentially distributed with rate
$\alpha>0$. When a customer's
waiting time in queue exceeds his
patience time, the customer leaves the system without service.
Once a customer enters service, he does not abandon.
The sequence of interarrival times, service times, and abandon times
are assumed to be independent.

To describe the \textquotedblleft state\textquotedblright\ of the
system as time evolves, we let $Z_{k}^{n}(t)$ denote the number of
customers \emph{in phase $k$ service} in the $n$th system at time
$t$; service times in phase $k$ are exponentially distributed with
rate $\nu_{k}$. We use $Z^{n}(t)$ to denote the corresponding
$K$-dimensional vector. We call $Z^{n}=\{Z^{n}(t):t\geq0\}$ the
\emph{server-allocation process}. Let $N^{n}(t)$ denote the number
of customers in the $n$th system at time $t$, either in queue or in
service.
Let $A^n(t)$ be the ``age'' of the interarrival time at time $t$,
i.e., the time between the arrival time of the most recent arrival before time $t$
and time $t$.   Then,
$(A^n,N^n, Z^n)$ is called the state processes for the $n$th system.
It follows that for each $n$ fixed the state process
$(A^n,N^n, Z^n)$ has state space ${\cal S} =\R_+\times \Z_+ \times \Z_+^K$.
As time $t$ goes on, $A^n(t)$ increases linearly while $(N^n(t) ,
Z^n(t)) $ remains constant. When $A^n(t)$ reaches the interarrival for
the next arrival, it
instantaneous jumps to zero. We adopt the convention that all
processes are right continuous on $[0, \infty)$, having left limits in
$(0, \infty)$. It follows that  $(A^n,N^n, Z^n)$ is a
piecewise deterministic  Markov process that conforms to
Assumption~3.1 of \cite{dav84}, and hence it is a strong Markov process
\cite[page 362]{dav84}.
Throughout the paper, we make the following assumptions.
\begin{assumption} \label{ass:expinterarr}
The interarrival times $\{\xi^n (i): i=1, 2, \ldots \}$ from the
second customer onwards have the following representation:
$\xi^n(i)= \frac{1}{\lambda^n} u(i)$, $i\ge 1$, where $\{u(i):i\ge
1\}$ is an i.i.d.\ sequence with $\E(u(i))=1$ and
$\E(u(i))^2<\infty$.
\end{assumption}

\begin{assumption} \label{ass:discretestab}
For each $n$, the Markov process $(A^n, N^n, Z^n)$
is positive Harris recurrent and has stationary distribution $\pi^n$.
\end{assumption}

\jdnote{
Assumption \ref{ass:expinterarr} appears to be restrictive.  A
consequence of the assumption is that the squared coefficient of
variation (variance divided by squared mean) of interarrival times
does not depend on the scaling parameter $n$.  All results in this
paper can be shown to continue to hold if we adopt a triangular array of random
variables to model the family of interarrival time sequences and
impose some uniform integrability condition on the second moment; see,
for example, (3.4) of \cite{bra98b}. As discussed in the introduction,
the major motivation of this paper is to justify the diffusion model
procedure in \cite{DaiHe13}, which applies to any queueing
system with a fixed arrival process that has no scaling parameter.
We adopt the current assumption for notational simplicity without
affecting the relevance of our results. }

For the definition of positive Harris recurrence of a Markov process
\jdnote{in Assumption \ref{ass:discretestab}}, see, for example,
\cite{dai95a}.  \jdnote{Since our model allows customer abandonment,
  for each fixed $n$, the queueing system should be stable in some
  sense. However, to prove \tonnote{that} the Markov process is positive Harris
  recurrent, some conditions on the interarrival time distribution are
  needed. To keep our paper focused on the interchange of limits, we
  do not pursue the best sufficient conditions for positive Harris
  recurrence.} Proposition \ref{pro:positiveCurrence} in
Appendix~\ref{sec:posit-harr-recurr} provides a sufficient condition
on the interarrival distribution for Assumption \ref{ass:discretestab}
to hold.

\subsection{Halfin-Whitt regime and diffusion limits}
\label{sec:halfin-whitt-regime}
We describe a parameter regime, known as the Halfin-Whitt regime, that
was first introduced in \cite{HalWhi81}. When the parameters are in this regime, certain
diffusion processes serve as good approximations for $GI/Ph/n+M$
systems.  Most of the materials in this section can be found in
\cite{DaiHeTezcan10} and readers are referred there for more details.

We consider a sequence of $GI/Ph/n+M$ systems indexed by $n$. The
arrival rate $\lambda^n$ depends on $n$.
 We assume that
the service time and the patience time distributions do not depend on
$n$. Let $1/\mu$  denote the mean service time. Since service
time distribution is assumed to be phase type with parameters $(p,
\nu, P)$ and the
$i$th component of $(I-P')^{-1}p$ is the expected number of visits to
phase $i$ before the  Markov chain is absorbed into state
$K+1$, the mean service time has the following formula
\begin{equation}
  \label{eq:meanExpression}
    \frac{1}{\mu} = \sum_{i=1}^K \frac{1}{\nu_i} [(I-P')^{-1}p]_i.
\end{equation}
In vector form,   $1/\mu= e' R^{-1}p$,
where $e$ is the (column) vector of ones and
\begin{equation}
 R=(I-P^{\prime})\diag(\nu). \label{eq:R}%
\end{equation}
Define
\[
\rho^{n}=\frac{\lambda^{n}}{n\mu} \quad \text{and} \quad
\beta^n=\sqrt{n} (1 - \rho^n).
\]
The quantity $\rho^n$ is said to be the traffic intensity of the
$n$th system. We assume that
\begin{equation}\label{eq:halfinwhitt}
\lim_{n\to\infty} \beta^n=\beta \quad\text{ for some
}\beta\in\mathbb{R}.
\end{equation}
When condition (\ref{eq:halfinwhitt}) is satisfied, the sequence of
systems is critically loaded,
and is said to be in the Quality-and-Efficiency-Driven (QED) regime
or the Halfin-Whitt regime.

 Setting
\begin{equation}
X^{n}(t)=N^{n}(t)-n\quad\ \text{for }t\geq0, \label{eq:X}%
\end{equation}
we call $X^{n}=\{X^{n}(t):t\geq0\}$ the \emph{centered total-customer-count
process} in the $n$th system. One can check that $(X^{n}(t))^{+}$ is
the number of customers waiting in queue at time $t$, and
$(X^{n}(t))^{-}$ is the number of idle servers at time $t$. Clearly,
\[
e^{\prime}Z^{n}(t)=n-(X^{n}(t))^{-}\quad\ \text{for }t\geq0.
\]
Now we define the diffusion-scaled state processes:
\begin{equation}
\tilde{A}^{n}(t)=\frac{1}{\sqrt{n}}A^{n}(t),\qquad \tilde{X}^{n}(t)=\frac{1}{\sqrt{n}}X^{n}(t)\quad\ \text{for }t\geq0.
\label{eq:tildeXn}%
\end{equation}
and
\begin{equation}\label{eq:tildeZn}
\tilde{Z}^{n}(t)=\frac{1}{\sqrt{n}}{( {Z}^{n}(t) -n\gamma )}\quad\
\text{for }t\geq 0, \end{equation} where
\begin{equation}
  \label{eq:gamma}
\gamma=\mu R^{-1}p,
\end{equation}
which is a $K$-dimensional vector since $\mu\in\R_+$.
It follows from (\ref{eq:meanExpression}) that
$\sum_{k=1}^{K}\gamma_{k}=1$. One interprets
$\gamma_{k}$ to be the fraction of busy servers in phase $k$ service.
We write $\tS^n$ for the state space of $(\tilde A^n,\tilde X^n,\tilde Z^n)$:
\[
\tS^n = \{(a,x,z) \in \R_+\times \R\times \R^K: \sqrt{n} x +n \in
\Z_+,\sqrt{n} z + n\gamma\in \Z_+^K, e'z+x^-=0\}.
\]

The following result  establishes a
many-server diffusion limit  for a sequence of  $GI/Ph/n+M$ systems
that operate in the Halfin-Whitt regime. Because the abandonment rate
$\alpha$ is positive, the limit process is positive recurrent.
\begin{proposition}
\label{thm:critical} Consider a sequence of $GI/Ph/n+M$ systems
satisfying Assumption \ref{ass:expinterarr} and
(\ref{eq:halfinwhitt}). Assume that $\xi^n(0)=0$ and
$(\tilde{X}^{n}(0),\tilde{Z}^{n}(0))\Rightarrow(\tilde{X}(0),\tilde{Z}%
(0))$ for a pair of random variables $(\tilde{X}(0),\tilde{Z}(0)).$
We then have
\begin{enumerate}
\item[(a)]
\begin{equation}\label{eq:tildeXZ}
(\tilde{X}^{n},\tilde{Z}^{n})\Rightarrow(\tilde{X},\tilde{Z})\quad\text{
as }n\rightarrow\infty,
\end{equation} where $(\tilde{X},\tilde{Z})$
is some continuous time Markov process on
\begin{equation}
\label{eq:deftildeS}
\tS = \{ (x,z)\in \R\times \R^K: e' z + {x}^-=0\}.
\end{equation}
\item[(b)] The Markov process
$({\tilde X},{\tilde Z})$ given in (\ref{eq:tildeXZ}) is positive
recurrent and has a unique stationary distribution $\pi$.
\end{enumerate}
\end{proposition}
\begin{proof}
Part (a)  follows  from  \cite{DaiHeTezcan10}.
It is assumed in \cite{DaiHeTezcan10} that as $n\rightarrow\infty$,
\begin{equation}
\tilde{E}^{n}    \Rightarrow\tilde{E}\quad\text{ for some Brownian
motion $\tilde{E}$,}
\label{eq:Eclt}%
\end{equation}
where
\begin{equation} \label{eq:tildeEn}
\tilde{E}^{n}(t)=\frac{1}{\sqrt{n}} [E^{n}(t)-\lambda^{n}t],
\end{equation}
with ${E}^{n}(t)$ being the number of arrivals in $(0, t]$. Under
Assumption \ref{ass:expinterarr} on the interarrival times, it
follows from the functional central limit theorem
that (\ref{eq:Eclt}) holds.

Part (b) is an immediate corollary of Theorem 3 in
\cite{diekergao:OU}.
\end{proof}



\section{The main result}
\label{sec:main-result}

In this section, we work under Assumption \ref{ass:discretestab}.
Let $\tilde \pi_n$ be  the stationary
distributions  of the diffusion-scaled state process
$(\tilde A^n, \tilde X^n, \tilde Z^n)$ defined in (\ref{eq:tildeXn}) and
(\ref{eq:tildeZn}).  
Now we state the main theorem of this paper.

\begin{theorem} \label{thm:tightness}
\gaonote{Suppose that Assumptions \ref{ass:expinterarr} and
\ref{ass:discretestab} and the many-server heavy traffic
condition (\ref{eq:halfinwhitt}) hold. Assume that the abandonment
rate $\alpha$ is strictly positive.
Then the sequence of probability distributions $\{\tilde \pi^n: n \ge 1 \}$ is tight.}
\end{theorem}

The following corollary states the validity of interchange of heavy-traffic
and steady-state limits. Because \cite{diekergao:OU} proved
  that $(\tilde X, \tilde Z)$ has a
unique stationary distribution, the corollary follows from Theorem \ref{thm:tightness}
by a standard argument; see \cite{GamarnikZeevi06} and
\cite{budhiraja2009stationary}.
\begin{corollary} \label{thm:interchange}
\gaonote{Under the assumptions of Theorem~\ref{thm:tightness},
the sequence of marginal distributions of  $\tilde \pi^n$  on
$(\tilde {X}^{n},\tilde Z^n)$
 converges weakly to $\pi$, where $\pi$ is the
unique stationary distribution of $(\tilde{X},\tilde{Z})$ and
$(\tilde{X},\tilde{Z})$ is the diffusion process from
Proposition \ref{thm:critical}.}
\end{corollary}




\section{Our Lyapunov function and a fluid model}
\label{sec:Lyapunovfluid}

In this section, we first introduce the Lyapunov function
that lies at the heart of this paper.
We then introduce a fluid model
associated with the sequence of $GI/Ph/n+M$ systems, and assert
that our Lyapunov function is a geometric Lyapunov function for the fluid model.
The proof of our main result relies on a comparison between the diffusion-scaled processes
and the fluid model.

To define our Lyapunov function, we first introduce a lemma on so-called common
quadratic Lyapunov functions. The lemma was established in Theorem~1
of \cite{diekergao:OU}. Recall that $p$ is a discrete probability
distribution on $\{1, \ldots, K\}$, representing the initial
distribution of the phase type service time distribution, $e$ is a
vector of ones, and $R$ is a matrix given in (\ref{eq:R}) and is
obtained from the parameters of the phase-type service time
distribution.
\begin{lemma} \label{lem:cqlf}
There exists a  $K\times K$ positive definite matrix $Q$ such that
\begin{align*}
&QR+R'Q \quad \text{is positive definite}, \\
&Q(I-pe')R+ R'(I - ep')Q \quad \text{is positive semi-definite}.
\end{align*}
\end{lemma}

Our Lyapunov function is the square root of a function $g=g_\beta$ that depends on
the ``slack parameter'' $\beta$ in (\ref{eq:halfinwhitt}).
It is defined as
\begin{eqnarray} \label{eq:gbeta}
g(x,z)= \left\{\begin{array} {ll}
(x+ \beta)^2 + \kappa (z+\beta \gamma)'Q(z+\beta \gamma), & \text{if} \quad \beta \ge 0,\\
(\alpha x + \mu \beta)^2 + \kappa (z+\beta \gamma)'Q(z+\beta
\gamma), & \text{if} \quad \beta<0,
\end{array}\right.
\end{eqnarray}
where $Q$ is given in Lemma~\ref{lem:cqlf}, $\gamma$ is defined in
(\ref{eq:gamma}) and $\kappa$ is a large positive constant to be
determined later.

The quadratic function in (\ref{eq:gbeta}) is slightly different from
the one used in \cite{diekergao:OU} in their Lyapunov function.
Our modification is needed because \cite{diekergao:OU} focused
on a $K$-dimensional state process, but the current paper focuses on the
degenerate $(K+1)$-dimensional state process $(\tilde X, \tilde Z)$
on the manifold $\tS$; both state processes are equivalent.
More importantly, the centering for $(x, z)$ in the two quadratic
terms is different. Our centering allows us to obtain a stronger property for the fluid model; see part (a) of Lemma~\ref{lem:gbetaflu}.

\gaonote{ In addition, it is clear that when $\beta \ge 0$, the function $g$ achieves its minimum value 0 at $(x^*, z^*) = (-\beta, -\beta \gamma) \in \tS$. For $\beta<0$, the function $g$ achieves its minimum value $\kappa \beta^2 \gamma' Q \gamma$ at $(\bar x^*, \bar z^*) = (-\frac{\mu \beta}{ \alpha}, 0) \in \tS$. This can be verified by rewriting $g(x,z)$ as follows:
\begin{eqnarray*} \label{eq:gbeta-re}
g(x,z) &=&  (\alpha x + \mu \beta)^2 + \kappa z'Qz +  \kappa \beta^2 \gamma' Q \gamma + 2 \kappa\beta z' Q\gamma \\
& = & (\alpha x + \mu \beta)^2 + \kappa z'Qz +  \kappa \beta^2 \gamma' Q \gamma + 2 \kappa\beta b z'e \\
& = & (\alpha x + \mu \beta)^2 + \kappa z'Qz +  \kappa \beta^2 \gamma' Q \gamma + 2 \kappa b (-\beta) \cdot x^{-},
\end{eqnarray*}
where the second equality follows from Lemma~\ref{lem:Qgammaequalse} with constant $b>0$, and in the last equality we have used the fact that $(x, z) \in \tS$.
 We remark that
 if $\beta<0$, both $\sqrt{g}$ and $\sqrt{g} - \sqrt{\kappa \beta^2 \gamma' Q \gamma}$ can serve as our Lypaunov function. Those two functions only differ by a constant and the function $\sqrt{g} - \sqrt{\kappa \beta^2 \gamma' Q \gamma}$ has a minimum value of zero on the manifold $\tS$. We use $\sqrt{g}$ as our Lyapunov function for the sake of clarity of the presentation.
}

We now introduce a fluid model associated with
$GI/Ph/n+M$ systems. This fluid  model is defined through a map, which
is established in a more general setting in \cite{DaiHeTezcan10}.
Recall the definition of $\tS$ in (\ref{eq:deftildeS}) and write
\[
\tT = \{ (u,v)\in \R\times \R^K: e' v = 0\}.
\]

\begin{lemma}
\label{lem:Phi}
Let $\alpha>0$.
\begin{enumerate}
\item[(a)] For each
$(u,v)\in\D(\tT)$, there exists a unique $(x,z)\in\D(\tS)$ such that
\begin{eqnarray}
x(t)&=&u(t)-\alpha\int_{0}^{t}(x(s))^{+}\,ds-e^{\prime}R\int_{0}%
^{t}z(s)\,ds,\label{eq:mapPhi1}\\
z(t)&=&v(t)-p(x(t))^{-}-(I-pe^{\prime})R\int_{0}^{t}z(s)\,ds
\label{eq:mapPhi2}%
\end{eqnarray}
for $t\geq0$.
\item[(b)] For each $(u,v)\in\D(\tT)$,
define $(x,z)=\Psi(u,v)\in\D(\tS)$, where
$(x,z)$ satisfies (\ref{eq:mapPhi1}) and (\ref{eq:mapPhi2}). The map
$\Psi$ is well-defined and is continuous when both the domain $\D(\tT)$ and
the range $\D(\tS)$ are endowed with the
standard Skorohod $J_{1}$-topology.
\item[(c)] The map $\Psi$ is Lipschitz
continuous in the sense that for any $T>0$, there exists a constant
$C=C(T)>0$ such that
\[
\|\Psi(y^{1})-\Psi(y^{2}) \|_T \leq C \|y^{1}-y^{2}\|_T \quad\text{
for any }y^{1},y^{2}\in\D(\tT).
\]
\item[(d)] The map $\Psi$ is positively homogeneous in the sense that
\[
\Psi(by)=b\Psi(y)\quad\text{ for each }b>0\text{ and each }y\in\D(\tT).
\]
\end{enumerate}
\end{lemma}

We now define the fluid counterpart of the diffusion-scaled state
processes $(\tilde X^n, \tilde Z^n)$.
Fix an initial state $(\tilde x(0), \tilde z(0))\in \tS$. For $t \ge 0$, we set
\begin{eqnarray} \label{eq:fluiduvn}
\tilde u(t)=\tilde x(0)-\mu \beta  t \quad \mbox{and} \quad
\tilde v(t)=(I-pe') \tilde z(0),
\end{eqnarray}
and, after noting that $(\tilde u(0),\tilde v(0))\in \D(\tT)$, set
\begin{eqnarray} \label{eq:xzfluid}
(\tilde x, \tilde z ) = \Psi ( \tilde u, \tilde v).
\end{eqnarray}
We call $(\tilde x, \tilde z)$ the fluid model starting from $(\tilde x(0),\tilde z(0))$.
The next lemma is a negative drift condition for the fluid model, and states that the
function $\sqrt{g}$ is a geometric Lyapunov function for the fluid model.
Appendix~\ref{sec:negdriftfluid} is devoted to its proof.

\begin{lemma} \label{lem:fluidexpdif}
\gaonote{Fix some $t_0>0$. There exists constants $C=C(t_0)>0$ and $\epsilon=\epsilon(t_0)\in(0,1)$
such that for each initial state $(\tilde x(0), \tilde z(0))\in \tS$, we have
\begin{equation*}
\sqrt {g(\tilde x(t_0),\tilde z({t_0}))}  - \sqrt {g(\tilde x(0), \tilde z(0))}
 \le C-\epsilon\sqrt {g(\tilde x(0), \tilde z(0))}.
\end{equation*} }
\end{lemma}


\section{Our Lyapunov function and diffusion-scaled processes}
\label{sec:Lyapunovdiff}

In this section, we present a negative drift condition for the diffusion-scaled processes,
and we briefly outline how this leads to our main result.
We use the same Lyapunov function $\sqrt{g}$ as in the fluid model.

We are now ready to formulate our negative drift condition for the diffusion-scaled processes.
Here and in the rest of this paper,
we adopt the notational convention that
\[\E_{(a, x,z)}[  \quad \cdot \quad  ] = \E [  \quad \cdot \quad
|\tilde A^n(0)=a, \tilde X^n(0)= x, \tilde Z^n (0) =z  ] \]
for each initial state $(a,x,z)\in\tS^n$.

\begin{lemma}\label{lem:abddrift}
Fix any  $t_0>0$.
Under the assumptions of Theorem \ref{thm:tightness},
there exists a nonnegative function $g$ on $\tS$, as well as two constants
$C=C(t_0)>0$ and  $\epsilon=\epsilon(t_0) \in (0,1)$,  such that
for each $n$ and each feasible initial state $(a,x,z)\in\tS^n$, we have
\begin{align} \label{ineq:driftineq}
\E_{(a,x,z)} \left[\sqrt{g (\tilde X^n(t_0), \tilde Z^n(t_0))} \right] - \sqrt{g(x,z)}
\le C - \epsilon \sqrt{g(x,z)}.
\end{align}
\end{lemma}
The function $\sqrt{g}$ satisfying (\ref{ineq:driftineq}) is
essentially a geometric Lyapunov function with a geometric drift size
$1-\epsilon$ and drift time $t_0$, see for instance Section~3 in
\cite{GamarnikZeevi06}.  Readers are referred to
\cite{GamarnikZeevi06} and \cite{MeynTweedie09} for more details on
the definition of a Lyapunov function and its application in deriving bounds
for stationary distributions of Markov processes.


The proof of the negative drift condition in Lemma~\ref{lem:abddrift} is lengthy and will be given
in Appendix~\ref{sec:negdriftdiff}. Assuming the lemma, the
proof of Theorem~\ref{thm:tightness} is standard. We end this section
by giving a sketch of the proof, which is almost identical to the proof of
Theorem~5 in \cite{GamarnikZeevi06}.

\begin{proof} [Proof sketch of Theorem~\ref{thm:tightness}]
In order to show that $\tilde \pi^n$ is tight,
since $\tilde A^n\Rightarrow 0$ as $n\to\infty$ and $g$ has
compact level sets (see Appendix~\ref{sec:negdriftfluid}),
it is sufficient to show that for any given $\delta>0$, there exists some large
constant $s$, such that for all $n$ sufficiently large,
\[
{P_{\tilde \pi^n} ( \sqrt{g( \tilde X^n(0), \tilde Z^n(0))}>s)}\le  \delta.
\]
By Markov's inequality, it suffices to show that, for some $C<\infty$,
\begin{equation}
  \label{eq:ExUniformBound}
\E_{\tilde \pi^n} \left[ \sqrt{g( \tilde X^n(0), \tilde
    Z^n(0))}\right] \le C/\epsilon,
\end{equation}
where $C$ and $\epsilon$ are constants in  (\ref{ineq:driftineq}). To prove (\ref{eq:ExUniformBound}),
 we use (\ref{ineq:driftineq}) in the following form:
\[
\epsilon \sqrt{g(x,z)}-C \le
\sqrt{g(x,z)} - \E_{(a,x,z)} [\sqrt{g (\tilde X^n(t_0), \tilde Z^n(t_0))} ].
\]
We argue that the right-hand side is nonpositive after taking the expectation
with respect to $\tilde \pi^n$, and (\ref{eq:ExUniformBound}) then follows.
For each integer  $k\ge 1$ and each $(x,z)\in \R\times \R^K$,
set $f_k(x, z) = \sqrt{g(x,z)} \wedge k$.
It can be checked that $f_k(x,z) - \E_{(a,x,z)} f_k ( \tilde X^n(t_0), \tilde
Z^n(t_0) )$ is bounded below by $-C$ for all $(a,x,z)\in \tS^n$.
Therefore, Fatou's Lemma can be applied and we deduce that
\begin{eqnarray*}
\lefteqn{\int_{\tS} \left[\sqrt{g(x,z)}- \E_{(a,x,z)}
 \sqrt{g( \tilde X^n(t_0), \tilde Z^n(t_0))} \right] d \tilde \pi^n(a,x,z)}\\
&\le& \mathop {\lim \inf}_{k \rightarrow \infty} \int_{\tS}
\left[f_k(x,z) - \E_{(a,x,z)} f_k ( \tilde X^n(t_0), \tilde Z^n(t_0) )
\right] d \tilde \pi^n(a,x,z) = 0,
\end{eqnarray*}
where the equality follows from stationarity of $\tilde \pi^n$.
\end{proof}

\section*{Acknowledgments}
JGD is supported in part by NSF grants CMMI-0727400, CMMI-0825840, CMMI-1030589, and CNS-1248117.
ABD is supported in part by NSF grant CMMI-1252878, and
he also gratefully acknowledges the hospitality of the Korteweg-de Vries Institute.

\appendix

\section{Positive Harris recurrence of $GI/Ph/n+M$ queues}
\label{sec:posit-harr-recurr}

In this appendix, we provide a sufficient condition on the
interarrival time distribution under which Assumption
\ref{ass:discretestab} holds.  To study the positive Harris
recurrence of a $GI/Ph/n + M$ queue, we fix the number of server
$n$, as well as the interarrival distribution and phase-type service
time distribution. Because $n$ is fixed, unlike in Section
\ref{sec:model-descr-prob}, here we drop the superscript $n$ in all
relevant quantities. Let $\{\xi(i):i=0, 1, 2, \ldots\}$ be the
sequence of interarrival times with $\{\xi(i):i=1, 2, \}$ being an
i.i.d.\ sequence and $\xi(0)$ being the arrival time of the first
customer after time $0$. We use $F$ to denote the cumulative
distribution of $\xi(1)$. We make the following assumptions on $F$:

\begin{enumerate}[(a)]
\item The distribution $F$ is unbounded, i.e., $F(x)<1$ for all
  $x>0$. \label{eq:unboundedsupp}



\item The distribution $F$ has a density, and furthermore the hazard function
\[
     h(x)= \frac{F'(x)}{1-F(x)} \quad \text{ for } x \ge 0
\]
of the distribution $F$ is locally bounded. \gaonote{Here, $h(0) =F'(0)$ where $F'(0)$ is interpreted as the right-derivative of $F$ at zero.}
\end{enumerate}

Recall that $\alpha$ is the rate of the exponential distribution for
patience times. For the definition of positive Harris recurrence, see, for
example, \cite{dai95a}. In the following proposition, $Q(t)$ is the
number of waiting customers in queue at time $t$. The other two
quantities $A(t)$ and $Z(t)$ retain the meaning in Section
\ref{sec:model-descr-prob}.

\begin{proposition} \label{pro:positiveCurrence}
  Suppose Assumptions (a)--(b) hold and $\alpha>0$. The Markov process
$\{(A(t), Q(t), Z(t)): t\ge 0\}$ is positive Harris recurrent.
\end{proposition}

\begin{proof}
 The process
$\{(A(t), Q(t), Z(t)): t\ge 0\}$ is known as a piecewise
deterministic Markov process as defined in \cite{dav84}; see also
Chapter~11 in \cite{rolski2009stochastic}. The main idea is to
construct a suitable Lyapunov function $f$ that is in the extended
domain of the generator $G$ of the Markov process. To construct the
Lyapunov function,  we first define the mean residual life function
$m(x)$ of the distribution function $F$ by
\[
 m(x)= \frac {1}{ 1 -F(x)} \cdot \int_{x}^{\infty}
(1-F(s))ds  = \E[\xi(1)-x| \xi(1)>x] \quad\text{ for } x \ge 0.
\]
We use
$(a,q,z)=(a, q, z_1, \ldots, z_K) \in {\cal S}=
\mathbb{R}_{+}\times \Z_+\times\mathbb{Z}_{+}^{K} $
to represent a state of the Markov process.  For each state $(a,q,z)$,
we define
\begin{equation} \label{eq:ln}
f(a,q,z)=2F(a)(1+m(a)) + q +\sum_{k=1}^K z_k.
\end{equation}
The first component, $F(a)( 1+m(a))$, is identical to a Lyapunov
function introduced in  Lemma 2.1 of \cite{lyarenewal}. We first
verify that $f$ is in the domain of the extended generator of the
Markov process (see Definition~5.2 in \cite{dav84} for the
definition of extended generator). We use Theorem~11.2.2 in
\cite{rolski2009stochastic} (see also Theorem~5.5 in \cite{dav84})
and check the three conditions there. Since the boundary set of the
piecewise deterministic Markov process is empty, and the sample path
of the age process $A$ is absolutely continuous between jumps, it
suffices to check that for each $t \ge 0$
\begin{equation} \label{eq:L1finite}
\E_{(a,q,z)} \left[ \sum_{i: T_i \le t} | f(A(T_i), Q(T_i), Z(T_i))
- f(A(T_i-), Q(T_i-), Z(T_i-))|  \right] < \infty,
\end{equation}
where $T_i$ are the jump epochs of the Markov process $\{(A(t),
Q(t), Z(t)): t\ge 0\}$. It follows from the proof of Lemma~2.1 and
Equation (2.5) in \cite{lyarenewal} that
\begin{equation*}
\E_{(a,q,z)} \left[ \sum_{i: T_i \le t} \left| F(A(T_i))(1+m(A(T_i)) -
F(A(T_i-)) (1+m(A(T_i-))\right| \right] < \infty.
\end{equation*}
Thus in order to establish (\ref{eq:L1finite}), we know from
(\ref{eq:ln}) that it is enough to show
\begin{equation} \label{eq:nbevent}
\E_{(a,q,z)} \left[ \sum_{i: T_i \le t} \left|  Q(T_i) + \sum_{k=1}^{K}
Z_k(T_i) - Q(T_i-) - \sum_{k=1}^{K} Z_k(T_i-) \right| \right] < \infty.
\end{equation}
To show (\ref{eq:nbevent}), we note that the number of
customer departures within $[0, t]$, due to either service completion
or abandonment, is bounded by $ q +n + E(t).$ Therefore,
\begin{displaymath}
\sum_{i: T_i \le t} \left|  Q(T_i) + \sum_{k=1}^{K}
Z_k(T_i) - Q(T_i-) - \sum_{k=1}^{K} Z_k(T_i-) \right| \le n+q + 2 E(t),
\end{displaymath}
from which (\ref{eq:nbevent}) follows because
\[
\E_{(a,q,z)} (E(t)) \le \lambda t + C (\sqrt{t} +1) \quad \text{for
some constant $C>0.$}
\]
See, e.g., Lemma~3.5 in \cite{budhiraja2006diffusion}.

Now we can write down the extended generator for the piecewise
deterministic Markov process $\{(A(t), Q(t), Z(t)): t\ge 0\}$. For
$\sum_k z_k=n$, $q>0$ and $a\ge 0$, it follows from (5.7) of
\cite{dav84} that
\begin{eqnarray} \label{eq:generator}
 Gf(a,q, z) &=&  \frac{\partial f}{\partial a}(a, q, z) + h(a) f(0, q+1, z) + \left(\alpha q +
 \sum_{k}\nu_k\right) f(a, q-1, z)  \nonumber \\ && {}- (\alpha q +
 \sum_{k}\nu_k + h(a)) f(a, q, z) \nonumber \\
&=& 2 \left[ [F(a) (1+m(a))]'-h(a)[F(a)(1+m(a))] \right] +h(a) -
\left(\alpha q+ \sum_{k}\nu_k\right),
\nonumber \\
& \le & 2(1+h(a)) \left[1- 2 F(a)+ \int_{a}^{ \infty }
  {(1-F(s))ds}\right] + h(a) -
 \left(\alpha q+ \sum_{k}\nu_k\right), \nonumber \\
 \end{eqnarray}
where in the second equality we have used the fact that $F(0)=0$, and
the  inequality follows from the derivation on page 170 of
\cite{lyarenewal}.  Because
\begin{equation*}
\lim_{a \rightarrow \infty} \left[1- 2 F(a)+ \int_{a}^{ \infty }
{(1-F(s))ds} \right] = -1,
\end{equation*}
there exists some $C_1>0$
such that
\begin{equation}
  \label{eq:lessneghalf}
  1- 2 F(a)+ \int_{a}^{ \infty }
{(1-F(s))ds} \le -\frac{1}{2} \quad \text{ for } a> C_1.
\end{equation}
Combining (\ref{eq:generator}) and (\ref{eq:lessneghalf}), we have
for $a>C_1$ and $q>0$
\begin{equation} \label{eq:alargeqlarge}
 Gf(a, q, z ) \le - 1 -\left(\alpha q + \sum_{k}\nu_k\right).
\end{equation}
 For any $a\ge 0$,  $q=0$ and $z\ge 0$, one can check that
 (\ref{eq:alargeqlarge})
 continues to holds with $\sum_{k}\nu_k$ in (\ref{eq:alargeqlarge})
 being replaced by $
 \sum_{k: z_k>0}\nu_k$.
Therefore, we have
for $a>C_1$, any $q\in \Z_+$, and any $z\in \Z_+^K$,
\begin{equation}
  \label{eq:Lyapunovlessminus1}
  Gf(a, q, z)  \le -1.
\end{equation}
Let
\begin{displaymath}
  H = \sup_{0 \le a \le C_1} \left\{
2(1+h(a)) [1- 2 F(a)+ \int_{a}^{ \infty }
  {(1-F(s))ds}] + h(a)  \right\},
\end{displaymath}
which is finite by Assumption (b) on $F$.
It follows from (\ref{eq:generator}) that for  $a\in [0, C_1]$ and
$q\ge C_2=(H+1)/\alpha$, (\ref{eq:Lyapunovlessminus1}) also holds.
It follows that
\begin{equation}
  \label{eq:finalLyapunov}
  Gf(a, q, z)\le -1 + H 1_{B}(a,q, z),
\end{equation}
where $B$ is the compact set given by
\begin{displaymath}
  B=\left\{ a\in [0, C_1], \, q\in [0,
C_2]\cap \Z_+, \, z\in \Z_+^K,\, 0\le \sum_{k}z_k \le n\right\}.
\end{displaymath}
  Since $F$ is assumed to have density, it is spreadout. Together with
Assumption~(\ref{eq:unboundedsupp}) on $F$, this implies that
the set $B$ is a closed petite set in the state space
${\cal S}=\mathbb{R}_{+}\times\Z_+\times \mathbb{Z}_{+}^{K} $;  see, e.g.,
the proof of  Lemma~3.7 of \cite{meydow94}. It follows from
(\ref{eq:finalLyapunov}) and Theorem~4.2 of \cite{meytwe93b} that
the Markov process $\{(A(t), Q(t), Z(t)): t\ge 0\}$ is positive
Harris recurrent.
\end{proof}

\section{Negative drift condition for the fluid model}
\label{sec:negdriftfluid}

It is the goal of this appendix to establish the negative drift condition for the fluid model,
i.e., to prove Lemma~\ref{lem:fluidexpdif}.
Throughout, we let $( \tilde x, \tilde z)$ be defined through (\ref{eq:xzfluid}), i.e.,
as output of the map $\Psi$ with input given in (\ref{eq:fluiduvn}).
The initial condition $(\tilde x(0),\tilde z(0))\in \tS$ is arbitrary.

We start with several auxiliary lemmas, which establish key properties of our Lyapunov function and
the fluid model. Their proofs are deferred to the end of this appendix.
The next lemma says that $\sqrt{g}$ is Lipschitz continuous.
\begin{lemma}
\label{lem:Lipschitz}
There exists a constant $C$ such that
\[
\left|\sqrt{g(x^1,z^1)}-\sqrt{g(x^2,z^2)}\right| \le
C |(x^1,z^1)-(x^2,z^2)|
\]
for any $(x^1,z^1),(x^2,z^2)\in \tS$.
\end{lemma}

The next lemma implies that $(\tilde x(t),\tilde z(t))$
has derivatives with respect to $t$ almost everywhere.
\begin{lemma}
\label{lem:Lipfluid}
The function $t\mapsto (\tilde x(t),\tilde z(t))$ is locally Lipschitz continuous.
\end{lemma}

We next formulate two important properties of the function ${g}$ with fluid model input.

\begin{lemma} \label{lem:gbetaflu}
Let $(\tilde x,\tilde z)$ be the fluid model on $\tS$ from Section~\ref{sec:Lyapunovfluid}.
\begin{enumerate}
\item [(a)]
For any $(\tilde x(0),\tilde z(0))\in \tS$, we have
\[
\frac{{dg(\tilde x(t), \tilde z (t))}}{{dt}} \le 0,
\]
whenever  $ {g (\tilde x(t), \tilde z(t))} $ is
differentiable at $t$.

\item [(b)] There are positive constants $c, C$ and $M$
such that, for any $(\tilde x(0),\tilde z(0))\in \tS$
such that $|(\tilde x(t), \tilde z(t)) | \ge M$
and such that $g(\tilde x(t),\tilde z(t))$ is differentiable at $t$,
we have
\[
- C \cdot g(\tilde x(t),\tilde z(t))  \le
\frac{dg(\tilde x(t),\tilde z(t))}{dt} \le
- c \cdot g(\tilde x(t),\tilde z(t)).
\]
\end{enumerate}
\end{lemma}

With these three lemmas at our disposal,
we are ready to prove the negative drift condition for the fluid model.

\begin{proof}[Proof of Lemma \ref{lem:fluidexpdif}]
Throughout this proof, when using constants from auxiliary lemmas, their subscript
denotes the lemma they originated from.
\newcommand{\CLip}{C_{\ref{lem:Lipschitz}}}
\newcommand{\Cgbeta}{C_{\ref{lem:gbetaflu}}}
\newcommand{\cgbeta}{c_{\ref{lem:gbetaflu}}}
\newcommand{\Mgbeta}{M_{\ref{lem:gbetaflu}}}

Since $g(-\beta,-\beta\gamma)=0$ or $g(-\mu\beta/\alpha,-\beta\gamma)=0$,
one obtains from Lemma~\ref{lem:Lipschitz} that there exists some $\CLip$ such that for arbitrary $(x,z)$
\begin{eqnarray} \label{eq:glessxz}
\sqrt{ g (x,z) } \le \CLip|(x,z)| +  \CLip\max \{ |(\beta, \beta
\gamma)|, |( {\mu \beta}/{\alpha}, \beta \gamma )|\} .
\end{eqnarray}
Let $\cgbeta,\Cgbeta,\Mgbeta$ be the constants defined in Lemma~\ref{lem:gbetaflu}, and set
\[
M =   \CLip \Mgbeta + \CLip \max \{ |(\beta, \beta \gamma)|, |( {\mu
\beta}/{\alpha}, \beta \gamma )|\}.
\]
It follows from (\ref{eq:glessxz}) that if
\begin{displaymath}
\sqrt{g(\tilde x(t), \tilde z (t))} \ge M,
\end{displaymath}
 we have $|(\tilde x(t), \tilde z (t))| \ge \Mgbeta$ and thus by the second part of
Lemma~\ref{lem:gbetaflu},
\begin{equation}
\label{eq:expdrift2} - \Cgbeta \le \frac{d}{dt} \ln g(\tilde x(t),\tilde z(t)) \le  - \cgbeta.
\end{equation}

Pick some $C'$ such that
\begin{eqnarray} \label{const:C5}
C'> M \cdot \exp \left( \Cgbeta t_0 /2\right).
\end{eqnarray}
 The constants $\epsilon$ and $C$ from the statement of the lemma can be chosen as follows:
\[
\epsilon = 1-\exp\left(-\frac 12 \cgbeta t_0\right), \quad\quad
C=\epsilon C',
\]
as we now verify.
If $\sqrt{g(\tilde x(0),\tilde z(0))} \le C'$, then we have,
by Lemma~\ref{lem:Lipfluid} and the first part of Lemma~\ref{lem:gbetaflu},
\[
\sqrt{g(\tilde x(t_0),\tilde z(t_0))} -\sqrt{g(\tilde x(0),\tilde z(0))} \le 0 = C-\epsilon C' \le C-\epsilon \sqrt{g(\tilde x(0),\tilde z(0))}.
\]
We next consider the case $\sqrt{g(\tilde x(0), \tilde z (0))} >C'$.
We want to show that in this case
\begin{equation} \label{eq:gblowerbound}
\sqrt{g(\tilde x(t), \tilde z (t))}\ge M \quad \text{for all
$t \in [0,t_0]$, }
\end{equation}
which, together with (\ref{eq:expdrift2}), implies that
\[
\sqrt{g(\tilde x(t_0),\tilde z(t_0))} \le \exp(-\cgbeta t_0/2)
\sqrt{g(\tilde x(0),\tilde z(0))} =(1-\epsilon) \sqrt{g(\tilde x(0),\tilde z(0))}.
\]

To establish (\ref{eq:gblowerbound}), we assume that $ \sqrt{g(\tilde x(0), \tilde z (0))} >C' $
and define
\[
\tau = \inf \{t \ge 0: \sqrt{g(\tilde x(t), \tilde z (t))} < M\}.
\]
One readily checks that $\tau >0$ and we now show that in fact $\tau > t_0$.
If $\tau = \infty,$ the claim is true.  Now we
assume $\tau < \infty$. For $t \in [0, \tau)$ we have
\begin{equation*}
\sqrt{g(\tilde x(t), \tilde z (t))}\ge M,
\end{equation*}
where we have used the definition $\tau$.
Now we can apply (\ref{eq:expdrift2}) for $t \in [0, \tau),$ and obtain
\[
\ln g(\tilde x(0),\tilde z(0))-
\Cgbeta t \le  \ln g(\tilde x(t),\tilde z(t)) \le  \ln
g(\tilde x(0),\tilde z(0)) - \cgbeta t.
\]
On combining this with the definition of $\tau,$ it follows that if
$\sqrt{g(\tilde x(0), \tilde z (0))}
>C',$
\begin{eqnarray*}
\ln M = \ln \sqrt{g(\tilde x( \tau ),\tilde z( \tau ))}
&\ge& \ln
\sqrt{g(\tilde x(0),\tilde z(0))}- \Cgbeta \tau/2 \\
&\ge& \ln C' - \Cgbeta \tau/2 \\
&> & \frac 12\Cgbeta (t_0 - \tau)+  \ln M,
\end{eqnarray*}
where we use (\ref{const:C5}) in the last inequality.
This shows that $\tau> t_0$, which proves (\ref{eq:gblowerbound})
and therefore the statement of the lemma.
\end{proof}

\subsection*{Proofs of auxiliary lemmas}
We now prove Lemmas~\ref{lem:Lipschitz}, \ref{lem:Lipfluid}, and \ref{lem:gbetaflu}.

\begin{proof}[Proof of Lemma~\ref{lem:Lipschitz}]
We first discuss the case $\beta \ge 0.$
In that case,  we have $g(x,z) = \| (x+\beta, z+\beta\gamma) \|^2_{Q}$,
where
\[
\| (x,z) \|^2_{Q}  = x^2 + \kappa z' Q z.
\]
Note that $\|\cdot\|_{ Q} $ defines a norm since $Q$ is positive definite.

Thus for $\beta \ge 0$, there exists a constant $C>0$ such that
\begin{eqnarray*}
\left|\sqrt{g(x_1,z_1)}-\sqrt{g(x_2,z_2)}\right| &=&
| \| (x_1+\beta, z_1+ \beta \gamma)\|_{ Q} - \| (x_2+\beta, z_2+
\beta \gamma)\|_{ Q} | \\
&\le& \| (x_1+\beta, z_1+ \beta \gamma)- (x_2+\beta, z_2+ \beta
\gamma)  \|_{ Q} \\
&= & \| (x_1 - x_2, z_1 -z_2) \|_{ Q} \\
&\le & C |(x_1,z_1) - (x_2,z_2)|,
\end{eqnarray*}
where the first inequality follows from the subadditivity property
of the norm $\| \cdot \|_{ Q}$, and the last inequality follows
from the equivalence of the $|\cdot|$-norm and the $\| \cdot \|_{ Q}$-norm
on the Euclidean space $\R^{K+1}$. We therefore obtain the claim when $\beta \ge 0$.
The claim for $\beta<0$ can be established similarly.
\end{proof}

\begin{proof}[Proof of Lemma~\ref{lem:Lipfluid}]
Since $t\mapsto (\tilde u(t), \tilde v (t) )$ is continuous in $t$, we deduce
from Lemma~9 in \cite{DaiHeTezcan10} that $t\mapsto ( \tilde x(t),
\tilde z (t))$ is also continuous in $t$.
It follows from (\ref{eq:mapPhi1}) that $t\mapsto \tilde x(t)$ is locally
Lipschitz continuous. Moreover, using the fact that for all $s, t\ge0$,
\[|{\tilde x (t)}^- -{\tilde x (s)}^- | \le |\tilde x (t)-\tilde x (s)
| ,\] we conclude from (\ref{eq:mapPhi2}) that $\tilde z(\cdot)$
is also locally Lipschitz continuous, hence $(\tilde x(\cdot),
\tilde z (\cdot))$ is locally Lipschitz.
\end{proof}

For the proof of Lemma~\ref{lem:gbetaflu}, we need to study derivatives of the fluid model.
From (\ref{eq:mapPhi1}), (\ref{eq:mapPhi2}),
(\ref{eq:fluiduvn}) and (\ref{eq:xzfluid}), it is straightforward to see
that when $ \tilde x(t) \ge 0$, we have $e'\tilde z(t)=0$ and
\begin{eqnarray} \label{eq:dynamicxpos}
\frac{d}{{dt}}\left[ \begin{array}{l}
 \tilde  x (t) \\
  \tilde z (t) \\
 \end{array} \right] =  \left( \begin{array}{l}
 -\mu \beta \\
  0  \\
 \end{array} \right)- \left( {\begin{array}{*{20}{c}}
   \alpha  & {e'R}  \\
   0 & {(I - pe')R}  \\
\end{array}} \right)\left[ \begin{array}{l}
 \tilde x (t)\\
  \tilde z (t) \\
 \end{array} \right].
 \end{eqnarray}
On the other hand, when $ \tilde x(t)<0$, we have $e' \tilde
z(t)=\tilde x(t), $ and
\begin{eqnarray} \label{eq:dynamicxneg}
\frac{d}{{dt}}\left[ \begin{array}{l}
 \tilde  x (t) \\
  \tilde z  (t) \\
 \end{array} \right] =
  \left( \begin{array}{l}
 -\mu \beta \\
   -\mu \beta p \\
 \end{array} \right) - \left( {\begin{array}{*{20}{c}}
   0 & {e'R}  \\
   0 & R  \\
\end{array}} \right)\left[ \begin{array}{l}
 \tilde x(t) \\
 \tilde z (t) \\
 \end{array} \right].
 \end{eqnarray}

We need two properties of the matrix $Q$
from Lemma~\ref{lem:cqlf}, which are recorded in the following lemma.

\begin{lemma}
\label{lem:Qgammaequalse}
Let $Q$ be a positive definite matrix from Lemma~\ref{lem:cqlf}.
\begin{enumerate}
\item[(a)] \gaonote{$Q\gamma = b e$ for some $b>0$.}
\item[(b)]
Up to a multiplicative constant, $\gamma$ is the only vector satisfying
\[
[Q (I-pe')R + R'(I-ep') Q]\gamma=0.
\]
\end{enumerate}
\end{lemma}
\begin{proof}
One directly verifies that $\gamma= \mu R^{-1} p$ satisfies
\[
\gamma'[Q (I-pe')R + R'(I-ep') Q] \gamma=0.
\]
Lemma~\ref{lem:cqlf} states that $Q(I-pe')R+ R'(I - ep')Q$ is a
positive semi-definite matrix, thus we deduce that
\[
[Q (I-pe')R + R'(I-ep') Q]\gamma =0,
\]
which immediately implies that $Q(I-pe')R \gamma=0$ by definition of $\gamma$.
Since $(I-pe')$ has a simple zero eigenvalue and $R$ is nonsingular, we must have
\begin{equation*}
Q\gamma=b e \quad \text{for some $b \ne 0$}.
\end{equation*}
In addition, left-multiplying $\gamma'$ at both sides of the above equality, we obtain that $b=\gamma'Q\gamma>0$, using the fact that $Q$ is positive definite.

As a corollary to part (a), we obtain $(I-ep') Q R^{-1} p = 0$, which we use in the proof of part (b):
it implies that
\begin{eqnarray*}
\lefteqn{Q (I-pe')R + R' (I-ep') Q} \nonumber \\
&&=  R' (I-ep')\cdot [ (R^{-1})' Q +  Q R^{-1}] \cdot (I-pe') R.
\end{eqnarray*}
Lemma~\ref{lem:cqlf} states that $QR+R'Q$ is positive definite.
After left-multiplying by $(R^{-1})'$ and right-multiplying by
$R^{-1}$, we find that $QR^{-1}+(R^{-1})'Q$ is also positive
definite. Since $(I-pe')R$ has rank $K-1$ and has a simple
zero eigenvalue with a right eigenvector $\gamma= \mu R^{-1} p$,
we obtain from the preceding display that
$\gamma= \mu R^{-1} p$ is the unique vector (up to a constant) such that
\[
\gamma'[Q (I-pe')R + R'(I-ep') Q]\gamma=0.
\]
This implies part (b) of the lemma.
\end{proof}

The next corollary controls the term involving $z$ in our Lyapunov function.

\begin{lemma}
\label{lem:boundsecondpart}
There exist constants $c,C>0$ such that, whenever the derivative exists,
\[
-C|\tilde z(t)|^2\le
\frac{d}{dt}  \left[(\tilde z(t)+\beta\gamma)'Q(\tilde z(t)+\beta\gamma) \right]
\le -c|\tilde z(t)|^2,
\]
as long as $e'\tilde z(t)=0$ or equivalently $\tilde x(t)\ge 0$.
\end{lemma}
\begin{proof}
It is readily seen with part (a) of Lemma~\ref{lem:Qgammaequalse} that
\begin{eqnarray*}
\frac{d}{dt}  \left[(\tilde z(t)+\beta\gamma)'Q(\tilde z(t)+\beta\gamma)\right]  &=&
2(\tilde z(t)+\beta\gamma)'Q \frac{d}{dt} \tilde z(t) \\
&=&\tilde z(t)'[Q (I-pe')R + R'(I-ep') Q] \tilde z(t).
\end{eqnarray*}
Since $e'\gamma=1\ne 0$, we deduce from Lemma~\ref{lem:cqlf} and
part (b) of Lemma~\ref{lem:Qgammaequalse} that the quadratic form $h'[Q (I-pe')R +
R'(I-ep') Q] h$ has a global (nonzero) minimum and maximum over the
compact set $\{h \in \R^K: e'h = 0, |h| =1\}$.
This yields the statement of the lemma.
\end{proof}

We are now ready to prove Lemma~\ref{lem:gbetaflu}.

\begin{proof}[Proof of Lemma~\ref{lem:gbetaflu}]
We discuss the cases $\beta \ge 0$ and $\beta<0$ separately.

\textbf{Case 1: $\beta \ge 0.$}
In this case, we obtain from (\ref{eq:gbeta}) that
\begin{equation}
\label{eq:startingpointpos}
\frac{d}{dt} g(\tilde x(t),\tilde z(t)) = \frac{d}{dt} (\tilde x(t)+\beta)^2 +
\kappa \frac{d}{dt}  \left[(\tilde z(t)+\beta\gamma)'Q(\tilde z(t)+\beta\gamma) \right].
\end{equation}
To compute the derivative with respect to $t$, we discuss two subcases.


\textbf{Case 1.1: $\beta \ge 0,$ $\tilde x(t) \ge 0. $}

In this case we have $e' \tilde z(t)=0$.
We use the differential equation (\ref{eq:dynamicxpos})
to rewrite the expression in (\ref{eq:startingpointpos}).
For the first term, this yields
\begin{equation} \label{eq:dgdtbeatxpos}
\frac{d}{dt} (\tilde x(t)+\beta)^2 = -2( \tilde x(t)+ \beta) (\alpha \tilde x(t) +\mu \beta + e'R \tilde z(t)).
\end{equation}
To bound this further, we use
our assumption that $\beta$ and $\tilde x(t)$ are nonnegative, which implies that
\[
 (\alpha \wedge \mu) (\tilde x(t)+\beta)  \le \mu \beta + \alpha \tilde x(t) \le (\alpha \vee \mu) (\tilde
  x(t)+\beta),
\]
where $\alpha \wedge \mu =\min\{ \alpha, \mu\}$ and $\alpha \vee \mu=\max\{ \alpha, \mu\}$.

We bound (\ref{eq:dgdtbeatxpos}) from above as follows.
Since $\alpha \wedge \mu>0$, we deduce that there exists some (large) $\kappa$ so that
\begin{eqnarray*}
|2( \tilde x(t)+ \beta) e'R \tilde z(t)| &\le& 2 |e'R| \cdot
(\tilde x(t)+ \beta) \cdot |\tilde z(t)|  \\
&\le& (\alpha \wedge \mu) (\tilde x(t)+ \beta)^2 +
\kappa\frac{c_{\ref{lem:boundsecondpart}}}{2} |\tilde z(t)|^2,
\end{eqnarray*}
where $c_{\ref{lem:boundsecondpart}}$ is the constant from Lemma~\ref{lem:boundsecondpart}.
We have thus obtained
\[
\frac{d}{dt} (\tilde x(t)+\beta)^2  \le - (\alpha \wedge \mu) ( \tilde x(t)+ \beta)^2
+ \kappa\frac{c_{\ref{lem:boundsecondpart}}}{2} |\tilde z(t)|^2.
\]
Combining this with Lemma~\ref{lem:boundsecondpart} and (\ref{eq:startingpointpos}), this yields
\[
\frac{d}{dt} g(\tilde x(t),\tilde z(t)) \le - (\alpha \wedge \mu) ( \tilde x(t)+ \beta)^2
- \kappa\frac{c_{\ref{lem:boundsecondpart}}}{2} |\tilde z(t)|^2,
\]
which establishes part (a) of the statement in Case 1.1.
It also gives the upper bound claimed in part (b), since the positive definiteness of $Q$ yields
a constant $c'>0$ such that
\[
|\tilde z(t)|^2 \ge c' \tilde z(t)'Q\tilde z(t) \ge \frac 12 c' (\tilde z(t)+\beta\gamma)'Q(\tilde z(t)+\beta\gamma),
\]
where the last inequality holds outside of some compact set.
To prove the lower bound claimed in part (b), we similarly note that
\begin{eqnarray*}
\frac{d}{dt} (\tilde x(t)+\beta)^2 &\ge& -2 (\alpha \vee \mu)
( \tilde x(t)+ \beta)^2
- 2(\tilde x(t)+\beta)e'R\tilde z(t) \\
&\ge& -[2 (\alpha \vee \mu)+(\alpha\wedge\mu)] ( \tilde x(t)+ \beta)^2
- \kappa\frac{c_{\ref{lem:boundsecondpart}}}{2} |\tilde z(t)|^2,
\end{eqnarray*}
and one can bound the second term in (\ref{eq:startingpointpos}) with Lemma~\ref{lem:boundsecondpart}.
We conclude that
\[
\frac{d}{dt} g(\tilde x(t),\tilde z(t)) \ge -[2 (\alpha \vee \mu)+(\alpha\wedge\mu)] ( \tilde x(t)+ \beta)^2
- \kappa\left(\frac{c_{\ref{lem:boundsecondpart}}}{2}+C_{\ref{lem:boundsecondpart}}\right) |\tilde z(t)|^2.
\]
By positive definiteness of $Q$ there exists some constant $C'>0$
such that, outside of some compact set,

\[
|\tilde z(t)|^2 \le C' \tilde z(t)'Q\tilde z(t) \le 2C' (\tilde z(t)+\beta\gamma)'Q(\tilde z(t)+\beta\gamma),
\]
and we have thus shown all the claims in case $\beta\ge0$ and $\tilde x(t)\ge 0$.

\textbf{Case 1.2: $\beta \ge 0,$ $\tilde x(t) < 0.$ }

In this case one has $e' \tilde z(t)=\tilde x(t) $. We use the differential equation
(\ref{eq:dynamicxneg}) to rewrite (\ref{eq:startingpointpos}).
This leads to
\begin{eqnarray*}
\frac{{dg(\tilde x(t),\tilde z(t))}}{{dt}}&=& 2( \tilde
x(t)+ \beta) (- \mu \beta -
e'R \tilde z(t)) \\
&&- \kappa(\tilde z(t)+ \beta \gamma)' [Q R +
R'Q](\tilde z(t)+ \beta \gamma)\\
&=&-2( \tilde z(t)+ \beta\gamma)' ee'R ( \tilde z(t)+ \beta \gamma) \\
&&- \kappa(\tilde z(t)+ \beta \gamma)' [Q R +
R' Q](\tilde z(t)+ \beta \gamma)\\
&=&-( \tilde z(t)+ \beta \gamma)' [e e'R +R'ee'+ \kappa(Q R + R' Q)](\tilde z(t)+ \beta \gamma),
\end{eqnarray*}
where we use $\gamma=\mu R^{-1} p$ and $ \tilde
x(t)+ \beta =e' ( \tilde z(t)+ \beta \gamma)$.
It follows from Lemma~\ref{lem:cqlf} that we can choose $\kappa$ large so that
$e e'R +R'ee'+ \kappa(Q R + R' Q)$ is positive definite.
This immediately yields part (a) of the statement in case 1.2.

By definition of $g$, again using $e'\tilde z(t)=\tilde x(t)$, we also have
\[
g(\tilde x(t),\tilde z(t) )=
 ( \tilde z(t) +\beta \gamma)'[e e' + \kappa Q]( \tilde z(t) +\beta \gamma).
\]
Since $ee'+\kappa Q$ is also positive definite, the proof for the case
$\beta \ge 0$ and $\tilde x(t) < 0$ is also complete.

\textbf{Case 2: $\beta < 0.$}

In this case, we obtain from (\ref{eq:gbeta}) that
\begin{equation}
\label{eq:startingpointneg} \frac{d}{dt} g(\tilde x(t),\tilde z(t))
= \frac{d}{dt} (\alpha \tilde x(t)+\mu \beta)^2 + \kappa
\frac{d}{dt} \left[(\tilde z(t)+\beta\gamma)'Q(\tilde
z(t)+\beta\gamma) \right].
\end{equation}
As in Case 1, we discuss two subcases.

\textbf{Case 2.1: $\beta < 0,$ $\tilde x(t) \ge 0.$ }
We use the differential equation (\ref{eq:dynamicxpos})
to rewrite the expression in (\ref{eq:startingpointneg}).
For the first term, this yields
\begin{equation*}
\frac{d}{dt} (\alpha\tilde x(t)+\mu\beta)^2
= -2 \alpha (\alpha  \tilde x(t)+ \mu \beta) (\alpha \tilde x(t)+\mu \beta +e'R \tilde z(t)).
\end{equation*}
The rest of the argument is almost identical to the one for Case~1.1.
Increasing $\kappa$ if necessary, one can show that
\[
|2\alpha(\alpha\tilde x(t)+\mu\beta) e'R\tilde z(t)| \le \alpha (\alpha\tilde x(t)+\mu\beta)^2 +
\kappa\frac{c_{\ref{lem:boundsecondpart}}}{2} |\tilde z(t)|^2.
\]
This leads to
\begin{eqnarray*}
\frac{d}{dt} g(\tilde x(t),\tilde z(t)) &\le& - \alpha ( \alpha \tilde x(t)+ \mu \beta)^2
- \kappa\frac{c_{\ref{lem:boundsecondpart}}}{2} |\tilde z(t)|^2, \\
\frac{d}{dt} g(\tilde x(t),\tilde z(t)) &\ge& -3\alpha ( \alpha\tilde x(t)+ \mu \beta)^2
- \kappa\left(\frac{c_{\ref{lem:boundsecondpart}}}{2}+C_{\ref{lem:boundsecondpart}}\right) |\tilde z(t)|^2.
\end{eqnarray*}
The first inequality immediately establishes part (a) for $\beta<0$, $\tilde x(t)\ge 0$.
For part (b), one uses these two inequalities with the same arguments as in Case~1.1.

\textbf{Case 2.2: $\beta<0,$ $\tilde x(t) < 0.$}
In this case we have $e' \tilde z(t)= \tilde x(t)$.
We use the differential equation (\ref{eq:dynamicxneg})
to rewrite the expression in (\ref{eq:startingpointneg}).
This leads to
\begin{eqnarray*}
\frac{d}{dt}  g(\tilde x(t),\tilde z(t))
&=&  -2 \alpha (\alpha  \tilde x(t)+ \mu \beta) (\mu \beta + e'R \tilde z(t)) \\
&& - \kappa(\tilde z(t)+ \beta \gamma)' [Q R + R'Q](\tilde z(t)+ \beta \gamma) \\
&=&-2 \alpha ( \alpha  \tilde z(t)+ \mu \beta \gamma)'e e'R (\tilde z(t)+ \beta \gamma)\\
&& - \kappa(\tilde z(t)+ \beta \gamma)' [Q R +R'Q](\tilde z(t)+ \beta \gamma),
\end{eqnarray*}
where we use $\tilde x(t)= e' \tilde z(t)$.
We next use an argument similar to the one used in Case~1.1.
Since $\beta<0$, $\tilde z(t)'e<0$, we have
\begin{eqnarray} \label{ineq:xbetaneg}
(\alpha \vee \mu) \cdot (\tilde z(t) + \beta\gamma)'e\le  (\alpha
\tilde z(t)+ \mu \beta\gamma)'e \le (\alpha \wedge \mu) \cdot
(\tilde z(t) + \beta\gamma)'e.
\end{eqnarray}
As a result, we obtain that
\begin{eqnarray*}
|2 \alpha ( \alpha  \tilde x(t)+ \mu \beta) e'R (\tilde z(t)+
\beta \gamma)| &\le&  -2\alpha|e'R|\cdot  |\tilde z(t)+ \beta \gamma|
(\alpha\tilde z(t)+\mu\beta\gamma)'e\\
&\le& 2\alpha (\alpha \vee \mu) |e'R| \cdot|\tilde
z(t)+\beta\gamma|^2.
\end{eqnarray*}
In view of (\ref{eq:startingpointneg}), we therefore find that
\begin{eqnarray*}
\frac{d}{dt} g(\tilde x(t),\tilde z(t)) &\le& 2\alpha (\alpha
{\vee}\mu) |e'R| \cdot|\tilde z(t)+\beta\gamma|^2
- \kappa(\tilde z(t)+ \beta \gamma)' [Q R + R'Q](\tilde z(t)+ \beta \gamma)\, \\
\frac{d}{dt} g(\tilde x(t),\tilde z(t)) &\ge& -2\alpha (\alpha
{\vee} \mu) |e'R| \cdot|\tilde z(t)+\beta\gamma|^2 -
\kappa(\tilde z(t)+ \beta \gamma)' [Q R +R'Q](\tilde z(t)+ \beta
\gamma).
\end{eqnarray*}
Increasing $\kappa$ if necessary so that $ -2\alpha
(\alpha{\vee}\mu) |e'R| I +\kappa [Q R +R'Q]$ is positive definite,
we readily obtain part (a) of the lemma.
For part (b), we need to make a comparison with $g$.
By definition of $g$ and (\ref{ineq:xbetaneg}), we obtain from $e'\tilde z(t) =\tilde x(t)$ that
\begin{eqnarray*}
g(\tilde x(t),\tilde z(t) ) &\le&
 ( \tilde z(t) +\beta \gamma)'[ (\alpha \vee \mu) e e' + \kappa Q]( \tilde z(t) +\beta \gamma) \\
g(\tilde x(t),\tilde z(t) ) &\ge&
 ( \tilde z(t) +\beta \gamma)'[ (\alpha \wedge \mu) e e' + \kappa Q](\tilde z(t) +\beta \gamma),
\end{eqnarray*}
which also yields part (b) of the claim in Case~2.2.

Combining the four cases, we have completed the proof of the lemma.
\end{proof}

\section{Negative drift condition for the diffusion-scaled processes}
\label{sec:negdriftdiff}

It is the goal of this appendix to establish the negative drift condition for the diffusion
scaled processes, i.e., to prove Lemma~\ref{lem:abddrift}. For this, we use the
negative drift condition for the
fluid model from Lemma~\ref{lem:fluidexpdif}.

Following Sections~4 and 5 in \cite{DaiHeTezcan10}, we can use the
map $\Psi$ in Lemma~\ref{lem:Phi} to represent the diffusion-scaled state processes given
in (\ref{eq:tildeXn}) and (\ref{eq:tildeZn}):
\begin{equation}
\label{eq:tildeXZPsi}
(\tilde X^n, \tilde Z^n ) = \Psi ( \tilde U^n, \tilde V^n),
\end{equation}
where the exact form of the processes $\tilde U^n$ and $\tilde V^n$
is not important at this point; they are specified in the proof of Lemma~\ref{lem:unifmombd} below.
In view of this identity,
we establish the negative drift condition for the diffusion-scaled processes
by comparing the diffusion-scaled inputs $\tilde U^n$ and $\tilde V^n$
of the map $\Psi$ with their fluid analogs $\tilde u^n$ and $\tilde v^n$,
and then leveraging the negative drift condition of the fluid model.

Our negative drift result allows the diffusion-scaled process $(\tilde A^n,\tilde X^n,\tilde Z^n)$ to
start from an arbitrary initial condition $(\tilde A^n(0),\tilde X^n(0),\tilde Z^n(0)) =(a,x,z)\in\tS^n$.
We initialize the fluid model with the same point $(x,z)$, i.e.,
$(\tilde x(0),\tilde z(0)) = (x,z)$.
As a result, the fluid model depends on $n$ through the state space $\tS^n$ of its initial point.
Throughout this appendix, we stress this dependence
by writing $(\tilde x^n,\tilde z^n)$ for the fluid model instead of $(\tilde x,\tilde z)$.
Similarly, we write $\tilde u^n$ and $\tilde v^n$ instead of $\tilde u$ and $\tilde v$,
as defined through (\ref{eq:fluiduvn}).

The following auxiliary lemma ensures that the inputs to $\Psi$ are close to their fluid analogs.
Its proof is deferred to the end of this appendix.

\begin{lemma}
\label{lem:unifmombd} Fix $t_0 \ge 0$. There exists a constant $C=C(t_0)>0$
such that for each $n$ large enough and each initial state
$(a,x,z)=(\tilde a^n(0), \tilde x^n(0), \tilde z^n(0))\in\tS^n$, we have
\begin{eqnarray}
    \E_{(a,x,z)} \bigl(\|{\tilde U^n}-\tilde u^n\|_{t_0}\bigr) &<&
    C + C \sqrt[4]{ g (x,z)}. \label{eq:Umoment} \\
   \E _{(a,x,z)}\bigl[\|{\tilde V^n}-\tilde v^n\|_{t_0} \bigr]  &<&
    C + C \sqrt[4]{ g (x,z)}. \label{eq:Vmoment}
\end{eqnarray}
\end{lemma}

With Lemma~\ref{lem:unifmombd} at our disposal,
we are ready to prove the negative drift condition for diffusion-scaled processes.

\begin{proof}[Proof of Lemma \ref{lem:abddrift}]
Throughout this proof, when using constants from auxiliary lemmas,
their subscript again
denotes the lemma they originated from.
\newcommand{\eflu}{\epsilon_{\ref{lem:fluidexpdif}}}
\newcommand{\Cflu}{C_{\ref{lem:fluidexpdif}}}
\newcommand{\Cuv}{C_{\ref{lem:unifmombd}}}
\newcommand{\CLip}{C_{\ref{lem:Lipschitz}}}
\newcommand{\CLipPsi}{C_{\ref{lem:Phi}}}

Let $n$ be large enough as in Lemma \ref{lem:unifmombd} and  let
$(a,x,z)=(\tilde a^n(0), \tilde x^n(0), \tilde z^n(0))\in \tS^n$.
We first note that
\begin{eqnarray} \label{eq:keybound}
\lefteqn{ \left|\sqrt{g( \tilde X^n(t_0),  \tilde
Z^n(t_0))}-\sqrt{g(
\tilde x^n(t_0),  \tilde z^n(t_0))}\right|}\nonumber \\
&\le & \CLip |(\tilde X^n(t_0),  \tilde Z^n(t_0))-( \tilde x^n(t_0),
\tilde z^n(t_0))| \nonumber \\
&=& \CLip | \Psi({ \tilde U^n}, { \tilde V^n})(t_0)- \Psi({\tilde u^n},
\tilde v^n)({t_0})| \nonumber  \\
 &\le & \CLip \| \Psi({ \tilde U^n}, { \tilde V^n})- \Psi({\tilde u^n}, \tilde v^n)\|_{t_0}\nonumber \\
 &\le & \CLipPsi({t_0})  \CLip [\|{ \tilde U^n}-
\tilde u^n\|_{t_0}+\|{\tilde V^n}- \tilde v^n\|_{t_0}],
\end{eqnarray}
where the first inequality follows from Lemma \ref{lem:Lipschitz},
the equality follows from the fact that $(\tilde X^n,  \tilde Z^n)= \Psi
(\tilde U^n,  \tilde V^n)$ and $(\tilde x^n,  \tilde z^n)= \Psi (\tilde u^n,  \tilde v^n)$,
and the last inequality follows from the
Lipschitz continuity of the map $\Psi$ as in part (c) of Lemma~\ref{lem:Phi}.

It follows from  (\ref{eq:keybound}), Lemma~\ref{lem:fluidexpdif},
and Lemma~\ref{lem:unifmombd} that
\begin{eqnarray}
  \lefteqn{ \E_{(a,x,z)}\left(\sqrt{g( \tilde X^n(t_0),
        \tilde Z^n(t_0))} - \sqrt{g(x, z)}\right)}\nonumber \\
 &\le& \E_{(a,x,z)}\left|\sqrt{g( \tilde X^n(t_0),  \tilde Z^n(t_0)) }-\sqrt{g(\tilde x^n(t_0),
 \tilde z^n(t_0))} \right|\nonumber \\
&& {}
+\left(\sqrt{g(\tilde x^n(t_0), \tilde z^n(t_0))} -\sqrt{g(x, z)} \right)\nonumber \\
 &\le & 2 \CLipPsi({t_0}) \CLip  \Cuv(1 + \sqrt[4]{g(x,z)} )
 + \Cflu -\eflu \sqrt {g(x, z)}.  \label{ineq:diffc5}
\end{eqnarray}
Pick some $C'>0$ such that, for any $b\in \R$ with $ b\ge C'$,
\[
2 \CLipPsi({t_0}) \CLip \Cuv \sqrt[4]{{b}} -\frac12 \eflu
{\sqrt {b}} \le 0.
\]
The constants $C$ and $\epsilon$ in the statement of the lemma are chosen as follows:
\[
C= {2}\CLipPsi({t_0}) \CLip  \Cuv(1 + \sqrt[4]{C'} ) + \Cflu, \quad\quad \epsilon=\frac 12 \eflu,
\]
and we now verify the statement of the lemma with these definitions.
If $(x,z)$ satisfies $g(x,z)<C'$, then the right-hand side of (\ref{ineq:diffc5}) is bounded from above by
\[
{2}\CLipPsi({t_0}) \CLip  \Cuv(1 + \sqrt[4]{C'} ) +\Cflu - \eflu \sqrt{g(x,z)} = C-2\epsilon \sqrt{g(x,z)}.
\]
If $(x,z)$ satisfies $g(x,z)\ge C'$, then it is bounded from above by
\[
{2}\CLipPsi({t_0}) \CLip  \Cuv + \Cflu -\frac 12 \eflu \sqrt {g} = C -\epsilon\sqrt{g(x,z)} -\CLipPsi({t_0}) \CLip  \Cuv\sqrt[4]{C'}.
\]
We have thus obtained (\ref{ineq:driftineq}) in both cases.
\end{proof}

\subsection*{Proof of auxiliary lemma}

We now prove Lemma~\ref{lem:unifmombd}.

\begin{proof}[Proof of Lemma~\ref{lem:unifmombd}]
In this proof, the constant $C$ is a generic constant independent of $n$, but may vary line from line.
Fix $t_0>0$.
We start by specifying the processes $\tilde U^n$ and $\tilde V^n$ for which (\ref{eq:tildeXZPsi}) holds.
Following (5.10) and (5.11) in \cite{DaiHeTezcan10}, we have
\begin{align}
&  \tilde U^{n}(t)=\tilde X^n(0) - \mu\beta^n t +\tilde{E}^{n}(t)+ e^{\prime}%
\tilde M^{n}(t)- \tilde G^n\left( \int_{0}^{t}( \bar
X^{n}(s))^{+}\,ds \right),
\label{eq:Un}\\
&  \tilde V^{n}(t)=(I-pe') \tilde Z^n(0)+\tilde{\Phi}^{0,n}(\bar B^{n}%
(t))+(I-pe^{\prime}) \tilde M^{n}(t). \label{eq:Vntilde}%
\end{align}
The processes $\tilde E^n, \tilde M^{n}, \tilde G^{n}, \tilde{\Phi}^{0,n}$, $\bar X^n$, and $\bar B^{n}$
are defined as follows, see Section~5.2 of \cite{DaiHeTezcan10}.
First, $\tilde E^n$ is given in
(\ref{eq:tildeEn}). For each $k = 1, \ldots, K$, let $S_k$ be a
Poisson process with rate $\nu_k$, and let $G$ be a Poisson process
with rate $\alpha$. For each $n\ge 1$ define the diffusion-scaled
processes:
\begin{eqnarray}\label{eq:servabdscaled}
 \tilde S_k^n (t)
=\frac{1}{\sqrt{n}} (S_k (nt)-n \nu_k t), \quad \tilde G^n (t)
=\frac{1}{\sqrt{n}} (G (nt)-n \alpha t) \quad t\ge 0.
\end{eqnarray}
Moreover, for
each $N \ge 1$ and $k = 0, \ldots, K$, define the routing processes
\[\Phi^k (N)= \sum_{j=1}^{N} \phi^k (j), 
\]
where $\{\phi^k (j): j=1, 2, \ldots \}$ are i.i.d.\ Bernoulli random
vectors with mean $p^k$. Here $p^0=p,$ and $p^k$ is the $k$th
column of matrix $P'$. For each $n\ge 1$,  set
\begin{eqnarray} \label{eq:Phikn}
\tilde{\Phi}^{k,n} (t)= \frac{1}{\sqrt{n}}
\sum_{j=1}^{\lfloor nt \rfloor } (\phi^k (j) -p^k) \quad t\ge 0,
\end{eqnarray}
where, for an $x\in \R$, $\lfloor x \rfloor$ is the largest integer that is
less than or equal to $x$.
Let $T_k^n (t)$ be the cumulative amount of service effort received
by customers in phase $k$ service in $(0, t ]$, $B^n (t)$ be the
cumulative number of customers who have entered into service by time
$t$. Then $\tilde M^n(t)$ is defined via
\begin{eqnarray} \label{eq:tildeMn}
\tilde{M}^{n}(t)=\sum_{k=1}^{K}\tilde{\Phi}%
^{k,n}(\bar{S}_{k}^{n}(\bar{T}_{k}^{n}(t)))-(I-P^{\prime})\tilde{S}^{n}%
(\bar{T}^{n}(t)),
\end{eqnarray}
where $\bar{S}^{n}(t)=S(nt)/n$ and $\bar{T}^{n}(t)=T_n(t)/n$ for
$t\geq0$. We also let $\bar{B}^{n}(t)=B^n(t)/n$ for $t\geq0$.
Finally, we have
\begin{displaymath}
\bar X^n(t) = \frac{1}{n} X^n(t),
\end{displaymath}
where $X^n(t)$ is given in (\ref{eq:X}). As in \cite{DaiHeTezcan10},
we have that
\[ \text{$\tilde E^n, \tilde S_1^n, \ldots,
\tilde S_K^n, \tilde{\Phi}^{0,n}, \ldots, \tilde{\Phi}^{K,n} $ and
$\tilde G^{n}$ are mutually independent.} \]

Having explained the meaning of all processes involved, we now proceed and prove the statement
of Lemma~\ref{lem:unifmombd}.
 By (\ref{eq:Un}), (\ref{eq:Vntilde}), and (\ref{eq:fluiduvn}) we have
\begin{align}
& \|{\tilde U^{n} -\tilde u^n}\|_{t_0} \le
\mu |\beta^n-\beta| t_0+
\|{\tilde{E}^{n}}\|_{t_0}+\sqrt{K}\|{\tilde{M}^{n}}\|_{t_0}+
\left|\left|\tilde G^n \left( \int_{0}^{t}({\bar X^{n}(s)})^{+}\,ds\right)\right|\right|_{t_0}, \label{eq:Uu}\\
 &\|{\tilde V^{n} -\tilde v^n}\|_{t_0} \le \|\tilde{\Phi}^{0,n}(\bar
{B}^{n})\|_{t_0} +(1+|p|\sqrt{K}) \|{\tilde{M}^{n}}\|_{t_0}. \label{eq:Vv}
\end{align}

We first establish (\ref{eq:Umoment}). The proof is similar to that
of Lemma~3.5 in \cite{budhiraja2006diffusion} and we only highlight
the main differences. We start with inequality (\ref{eq:Uu}). We bound the
three terms in the right side of (\ref{eq:Uu}) separately. Let
$\{{E}^*(t): t \ge 0\}$ be a renewal process associated with i.i.d.\
sequence $\{u(i):i\ge 1\}$ that was given in
Assumption~\ref{ass:expinterarr}, where $u(1)$ is the first renewal
time. It follows from Assumption~\ref{ass:expinterarr} that for each
$t \ge 0,$
\begin{eqnarray} \label{eq:Etransform}
E^*(\lambda^n t) \le E^n(t) \le 1 + E^*(\lambda^n t).
\end{eqnarray}
Since $E^*(\cdot)$ is independent of $(A^n(0), X^n(0), Z^n(0))$ for
any $n,$ using the definition of ${\tilde{E}^{n}}$ in (\ref{eq:tildeEn})
and Equation (\ref{eq:Etransform}), we deduce that there exists a
constant $C$ independent of $n$ {and of initial states $(a,x,z)$}
such that
\begin{eqnarray*}
\E_{(a,x,z)} \|{\tilde{E}^{n}}\|_{t_0}^2 &\le& \E_{(a,x,z)}
 \sup_{0\le t \le t_0}  \frac{1}{n} \max\{|1+ E^*(\lambda^n t) - \lambda^n t |^2, | E^*(\lambda^n t) - \lambda^n t |^2\} \nonumber \\
&\le & \E_{(a,x,z)}  \sup_{0\le t \le t_0}  \frac{2}{n} (|
E^*(\lambda^n t) - \lambda^n t |^2 +1)\\
&=& \E  \sup_{0\le t \le t_0}  \frac{2}{n} (|
E^*(\lambda^n t) - \lambda^n t |^2 +1) \\
&= & 2 \E \sup_{0 \le t \le t_0} \frac{1}{\lambda^n}
|E^*(\lambda^n t) - \lambda^n t |^2 \cdot \frac{\lambda^n}{n} + \frac{2}{n} \nonumber \\
&\le & (t_0+1)C,
\end{eqnarray*} where the last inequality follows from Lemma~3.5 in
\cite{budhiraja2006diffusion} and the fact that the sequence
$\{{\lambda^n}/{n}: n \ge 1\}$ is bounded. Hence we have
\begin{eqnarray} \label{eq:boundE}
\E_{(a,x,z)} \|{\tilde{E}^{n}}\|_{t_0} \le C + C \sqrt{t_0}.
\end{eqnarray}

Moreover, we know from (\ref{eq:tildeMn}) and Lemma~3.5 in
\cite{budhiraja2006diffusion} that
\begin{eqnarray*}
\E_{(a,x,z)} \|{\tilde{M}^{n}}\|_{t_0}^2  \le { ({t_0}+1) C},
\end{eqnarray*}
where we use the fact that for each $k,$ $\bar{T}_k^n (t) \le t$ for
all $t \ge 0.$ This immediately implies
\begin{eqnarray}\label{eq:Mnmoment}
\E_{(a,x,z)} \|{\tilde{M}^{n}}\|_{t_0}  \le C + C \sqrt{t_0},
\end{eqnarray}

Finally we show there is a constant $C(t_0)$ which depends on $t_0$
but is independent of $n$ and any initial state $(a,x,z)\tS^n$ such that
\begin{eqnarray} \label{eq:boundG}
\E_{(a,x,z)}\left\|\tilde G^n \left( \int_{0}^{t}(\bar
X^{n}(s))^{+}\,ds\right)\right\|_{t_0} \le
 C(t_0) + C(t_0) \frac{\sqrt[4]{g(x,z)} } {\sqrt{n}}.
\end{eqnarray}
To see this, we know from (\ref{eq:X}) that for all $0 \le s \le t$
\begin{equation*}
(\bar{X}^{n}(s))^+\leq(\bar{X}^{n}(0))^{+}+\bar{E}^{n}(t) =
x^+/\sqrt{n}+\bar{E}^{n}(t),
\end{equation*}
where $x = \tilde X^n(0) = \sqrt{n} \bar X^n(0)$ and $\bar{E}^{n}(t)
= \frac{1}{n} {E}^{n}(t)$ for $t \ge 0.$ In conjunction with
(\ref{eq:Etransform}) we obtain
\begin{eqnarray*}
\int_{0}^{t}{(\bar{X}^{n}(s))^{+}\,ds} \le t\Bigl(\frac{
|x|}{\sqrt{n}}+\bar{E}^{n}(t)\Bigr) \le  t \frac{|x|}{\sqrt{n}}+
\frac{t}{n}{E^*}(\lambda^n t) + \frac{t}{n} .
\end{eqnarray*}
We have on each sample path
\begin{eqnarray}\label{eq:boundxt}
\sup_{0 \le t \le t_0} \left|\tilde G^n \left( \int_{0}^{t}(\bar
X^{n}(s))^{+}\,ds\right)\right| \le  \|\tilde G^n \|_{ \frac{ t_0|x|}{\sqrt{n}}+
\frac{t_0}{n}{E^*}(\lambda^n t_0) + \frac{t_0}{n}   }.
\end{eqnarray}
Since $G(\cdot)$ is a Poisson process, $\tilde G^n$
defined in (\ref{eq:servabdscaled}) is a martingale. In addition,
the random variable $E^*( \lambda^n t_0)$ is independent of $\tilde
G^n$ and the age $A^n(0)$ associated with the arrival process $E^n$.
Furthermore, it follows again from  Lemma~3.5 in
\cite{budhiraja2006diffusion} that there is some constant $C$
independent of $n,$ such that for $t \ge 0,$
\begin{eqnarray} \label{eq:boundElamdan}
\E \left[E^*(\lambda^n t) \right]\le \lambda^n t + C
\sqrt{\lambda^n} (1 + \sqrt{t}).
\end{eqnarray}
Therefore we deduce from (\ref{eq:boundxt}) that
\begin{eqnarray*}
\lefteqn{
\E_{(a,x,z)}\left\|\tilde G^n \left( \int_{0}^{t}(\bar
X^{n}(s))^{+}\,ds\right)\right\|_{t_0}^2 \le  \E_{(a,x,z)} \left[ \|\tilde G^n
\|_{
\frac{ t_0|x|}{\sqrt{n}}+ \frac{t_0}{n}{E^*}(\lambda^n t_0) + \frac{t_0}{n}}^2 \right]} \\
&= & \E\left[  \|\tilde G^n \|_{ \frac{ t_0|x|}{\sqrt{n}}+
\frac{t_0}{n}{E^*}(\lambda^n t_0) + \frac{t_0}{n}}^2 \right] \\
&=& \E \left[\E \left[\left.\|\tilde G^n \|_{ \frac{ t_0|x|}{\sqrt{n}}+
\frac{t_0}{n}{E^*}(\lambda^n t_0) + \frac{t_0}{n}}^2 \right|
{E^*}(\lambda^nt_0)\right]
\right] \\
&\le& \E \left[ 4 \alpha \left(\frac{ t_0|x|}{\sqrt{n}}+
\frac{t_0}{n}{E^*}(\lambda^n t_0) + \frac{t_0}{n}\right)  \right]    \\
&\le& 4 \alpha \frac{ t_0|x|}{\sqrt{n}} +  4 \alpha \frac{t_0}{n}
\left( {\lambda^n t_0} + C \sqrt{\lambda^n} + C \sqrt{\lambda^n t_0}
 + 1 \right),
\end{eqnarray*}
where in the second last inequality we use Doob's maximal
inequality {for martingales} (see, e.g., Theorem II.1.7
in~\cite{revuz1999continuous}), and in the last inequality we use
(\ref{eq:boundElamdan}). Since there is a constant $C$ independent
of $n$ such that the sequence $\{{\lambda^n}/{n}: n \ge 1\}$ is
bounded by $C$ and
\[ |x| \le \sqrt{g(x,z)} +C \quad \text{ for }  (a, x,z)\in \tS^n,  \]
we obtain (\ref{eq:boundG}). On combining (\ref{eq:boundE}),
(\ref{eq:Mnmoment}) and (\ref{eq:boundG}), we deduce from
(\ref{eq:Uu}) that (\ref{eq:Umoment}) holds.

To prove (\ref{eq:Vmoment}), we start  from inequality (\ref{eq:Vv}).
Since $B^n (t)$ is the cumulative number of customers who have
entered into service by time $t$, we obtain that
\[\bar{B}^{n}(t) = \frac{1}{n}B^n (t) \leq \frac{1}{n}({X}^{n}(0))^{+}+ \frac{1}{n}{E}^{n}(t).\]
In addition, it is clear that $\tilde{\Phi}^{0,n}$ defined in
(\ref{eq:Phikn}) is a martingale, thus we can proceed in a similar
fashion as the proof for (\ref{eq:boundG}) and show that there is a
constant $C(t_0)$ depending on $t_0,$ but is independent of $n$ and
any initial state $(a,x,z)$ such that
\begin{eqnarray*}
\E_{(a, x,z)}\|\tilde{\Phi}^{0,n}(\bar {B}^{n})\|_{t_0} \le
 C(t_0) + C(t_0) \frac{\sqrt[4]{g(x,z)} } {\sqrt{n}}.
\end{eqnarray*}
On combining (\ref{eq:Mnmoment}) we obtain (\ref{eq:Vmoment}) from
(\ref{eq:Vv}).
\end{proof}

%

\bibliographystyle{ims}
\bibliography{gao_ref}

\end{document}